\documentclass[journal]{IEEEtran}
\usepackage[T1]{fontenc}

\usepackage{amsmath,amssymb,amsthm,bm,mathtools,mathrsfs}
\usepackage{booktabs,array,graphicx,cite}
\usepackage[hidelinks]{hyperref}
\hypersetup{pdftitle={Global Output Lifting and Surjectivity in Invertible Nonlinear MIMO Systems},pdfauthor={Zhaobo Liu},pdfsubject={Nonlinear inversion and global output geometry},pdfkeywords={MIMO, surjectivity, detectability, nonlinear inversion}}
\newtheorem{theorem}{Theorem}
\newtheorem{lemma}{Lemma}
\newtheorem{proposition}{Proposition}
\newtheorem{corollary}{Corollary}
\newtheorem{assumption}{Assumption}
\theoremstyle{definition}
\newtheorem{definition}{Definition}
\newtheorem{remark}{Remark}
\newcommand{\R}{\mathbb R}
\newcommand{\col}{\operatorname{col}}

\newcommand{\id}{\operatorname{id}}
\newcommand{\KL}{\mathcal{KL}}
\newcommand{\Kinf}{\mathcal K_\infty}

\newcommand{\BB}{\mathcal B}
\newcommand{\DD}{\mathcal D}

\newcommand{\jet}[2]{j^{#1}#2}
\DeclareMathOperator{\diag}{diag}
\DeclareMathOperator{\im}{im}
\allowdisplaybreaks[1]

\begin{document}
\bstctlcite{BSTcontrol:ShowAuthors}
\title{Global Output Lifting and Surjectivity in Invertible Nonlinear MIMO Systems}
\author{Zhaobo~Liu%
\thanks{This work was supported in part by the National Natural Science Foundation of China under Grant 12401585, the Guangdong Basic and Applied Basic Research Foundation under Grant 2024A1515011542, and the General Program of Shenzhen Natural Science Foundation under Grant JCYJ20250604181037012.}
\thanks{Zhaobo Liu is with the Institute for Advanced Study, Shenzhen University, 3883 Baishi Road, Nanshan District, Shenzhen 518060, China (e-mail: liuzhaobo@szu.edu.cn).}}
\maketitle

\begin{abstract}
Global inversion of nonlinear multivariable systems raises the question of whether every target value of the selected output derivatives can be realized by a plant state. We prove that, under globally regular modified inversion with fixed row selection and finite termination at full input rank, every fixed homogeneous linear output assignment yields a surjective output-derivative map with a smooth global section. The result requires a state bound on each finite interval in terms of the initial state and finitely many output derivatives, uniformly over inputs. This condition is strictly weaker than weak uniform zero-state detectability, making the separate surjectivity requirement in a global inversion approach to stabilization a consequence of detectability under the stated inversion hypotheses. The section selects a realizing state smoothly for every derivative target, including in the presence of internal dynamics, and is constructed through complete lifting of compatible smooth references. For known plants, the same lifting method yields a smooth dynamic controller using full state feedback and current reference derivatives, with initialization from the plant state. From any initial state, it achieves exact tracking of any smooth desired output after any prescribed positive time.
\end{abstract}
\begin{IEEEkeywords}
Nonlinear systems, multivariable systems, inversion, detectability, global stabilization, surjectivity.
\end{IEEEkeywords}

\section{Introduction}
\IEEEPARstart{S}{tabilization}, trajectory tracking, and output regulation are central tasks of nonlinear control. For invertible multivariable systems, inversion turns prescribed output dynamics into feedback laws. This approach supports tracking and response synthesis \cite{HirschornTracking1981,IR1984,RAO2006}, stabilization \cite{Liberzon2004}, and robust dynamic feedback linearization and output regulation \cite{Wu2020,Wu2021,IsidoriWu2023,IsidoriJSSC2025}. Extending these designs across the state space requires understanding both the output behavior that the plant can realize and the internal motion that accompanies it.

Classical linear \cite{Silverman1969} and nonlinear \cite{HirschornSIAM1979,Hirschorn1979,Singh1981,Singh1982} inversion algorithms differentiate output equations to expose independent input directions, with rank invariants and intrinsic regularity describing the supporting conditions \cite{DBGM1989,DBG1990}. Under suitable regularity and termination conditions, they locally realize a compatible reference. For a homogeneous output assignment, realizing arbitrary derivative targets also requires the associated coordinate map to be onto the assigned derivative space. This property is a separate assumption in a global inversion theorem for stabilization \cite[Th. 2]{Liberzon2004}. Finding conditions that imply it is the surjectivity question highlighted in \cite[slide 119]{Isidori2017}. The same range requirement appears in a recent overview \cite[Sec. 6.1]{IsidoriJSSC2025} and recent stabilization and regulation constructions \cite{IsidoriNOLCOS2025,IsidoriCDC2025,Isidori2026}.

One approach to establishing surjectivity is to recover the entire state from the assigned derivative coordinates \cite{Wang2015,Wang2016}. A global diffeomorphism result is available for a class of systems with a global normal form, trivial zero dynamics, and triangular dependence of its multipliers \cite{Wang2017}. The resulting coordinate map coincides with the map used in the stabilization theorem \cite[Remark 5]{Wang2017}. The triangular structure permits successive recovery of the state, yielding a unique state for every coordinate target. With internal dynamics, the assigned derivatives may no longer determine the entire state, although every coordinate target may still have a preimage. The question is how to guarantee surjectivity when full state recovery from the assigned coordinates is unavailable.

We use output tracking to find a state realizing each derivative target. The construction selects a reference with prescribed terminal derivatives and uses inversion to generate a tracking trajectory whose endpoint realizes the target. The key difficulty is to ensure that this trajectory reaches the terminal time. Even when inversion is globally regular and the output and its required derivatives remain bounded, the internal state may escape in finite time. Stability and detectability estimates \cite{SW1997,LMS2002} and unboundedness observability \cite{AS1999} provide tools for controlling this state growth. We use a bound that controls the state on each finite interval in terms of the initial state and finitely many output derivatives, uniformly over inputs, ensuring that the inverse trajectory continues to the terminal time.

We prove that, under globally regular modified inversion with fixed row selection and finite termination at full input rank, the state bound above ensures surjectivity of the output-derivative map associated with every fixed homogeneous linear output assignment. Moreover, the map admits a smooth global section, so a realizing state can be selected smoothly for every derivative target even when internal dynamics remain. The state bound is strictly weaker than weak uniform zero-state detectability. Thus, under the stated inversion conditions, detectability implies the surjectivity property separately assumed in the global inversion approach to stabilization. This answers the surjectivity question in \cite{Isidori2017} within that framework.

Complete lifting also turns the existence analysis into a tracking design. For known plants, we construct a smooth dynamic controller that, from any initial state, exactly tracks any smooth desired output after any prescribed positive time. The controller generates a reference compatible with the initial state and smoothly joins it to the desired output, thereby accommodating desired outputs that do not satisfy the plant's initial compatibility conditions. It uses full state feedback and current reference derivatives and is initialized from the plant state. Under weak uniform zero-state detectability, bounded initial states and uniformly bounded desired outputs and required derivatives further ensure uniform state and input bounds for all time. These guarantees accommodate internal dynamics and do not require the original outputs to have a uniform vector relative degree.

Section~\ref{sec:setting} states the hypotheses, and Section~\ref{sec:main} presents the main results. Sections~\ref{sec:liftproof} and~\ref{sec:geometry} prove lifting and the geometric conclusions. Section~\ref{sec:consequences} treats feedback applications, and Section~\ref{sec:examples} presents a criterion for the state bound, examples, and numerical verification.

\noindent\textit{Notation.} The operators $\col$ and $\diag$ denote vertical stacking of vectors or matrices and block diagonal assembly, respectively. We write $Dg$ for the Jacobian of a smooth map $g$ and $D_zg$ for its partial Jacobian with respect to $z$. The symbols $I_d$ and $\id_{\R^d}$ denote the identity matrix and identity map, respectively. Superscripts $\top$ and $T$ denote transpose. We use the Euclidean norm $|\cdot|$ and its induced matrix norm $\|\cdot\|$. For a signal $v$, write $\|v\|_{[0,t]}\triangleq \sup_{0\le s\le t}|v(s)|$. For a smooth signal $y$, its output jet $\jet{N}{y}\triangleq \col(y,\dot y,\ldots,y^{(N)})$ collects derivatives of orders zero through $N$. For smooth maps of compatible dimensions, $L_fh\triangleq Dh\,f$ and $L_Gh\triangleq Dh\,G$ denote differentiation along a vector field and along the columns of a matrix field, respectively. A continuous function $\alpha:[0,\infty)\to[0,\infty)$ belongs to $\mathcal K$ if it is strictly increasing and $\alpha(0)=0$; it belongs to $\Kinf$ if it is also unbounded. A continuous function $\beta:[0,\infty)^2\to[0,\infty)$ belongs to $\KL$ if $\beta(\cdot,t)\in\mathcal K$ for each $t$ and $\beta(r,\cdot)$ is nonincreasing and tends to zero for each fixed $r$.

\section{System Description and Assumptions}\label{sec:setting}
We impose inversion conditions and a state bound to realize output references compatible with the initial state, with compatibility defined in Section~\ref{sec:main}. The inversion conditions determine the input from the state and output derivatives. The state bound then ensures that the resulting trajectory continues throughout any prescribed finite interval.

\subsection{System and inversion notation}
Consider the input-affine system with equal numbers of inputs and outputs
\begin{equation}\label{eq:plant}
 \dot x=f(x)+G(x)u,\qquad y=h(x),
\end{equation}
where $x\in\R^n$, $u,y\in\R^m$, and all maps are $C^\infty$ on their Euclidean domains. Solutions are understood on their maximal intervals unless stated otherwise. Forward completeness is not assumed. We work with smooth inputs. Any input defined on a neighborhood of a compact time interval may be smoothly extended beyond that interval.

Inversion determines the input from the state and prescribed output derivatives. The first derivative satisfies $\dot y=L_fh(x)+L_Gh(x)u$, which determines $u$ whenever $L_Gh(x)$ is nonsingular. Otherwise, higher derivatives may provide additional independent input equations. For example, consider
\[
 \dot x_1=u_1,\quad \dot x_2=x_3+u_1,\quad
 \dot x_3=u_2,\qquad y=(x_1,x_2)^\top.
\]
The first derivatives give $\dot y_1=u_1$ and $\dot y_2=x_3+u_1$. Eliminating $u_1$ yields $\dot y_2-\dot y_1=x_3$, and differentiation then gives $u_2=\ddot y_2-\ddot y_1$. Thus one input equation is obtained at the first derivative order and a second at the next order.

The nonlinear inversion methods of Hirschorn \cite{Hirschorn1979} and Singh \cite{Singh1981} provide the basis for this procedure. We use a modified inversion algorithm with fixed row selection that repeats the elimination and differentiation above. At each stage it selects equations whose input coefficient rows add independent directions, retains all previously selected equations, and uses the full retained set for elimination.

Let $\rho_k$ be the cumulative number of independent input equations after stage $k$. Under the global conditions below, these numbers are constant and satisfy $0=\rho_0\le\rho_1\le\cdots\le\rho_K=m$ for a finite $K$. A fixed permutation groups the outputs whose equations are first selected at stage $k$ into $y_{[k]}\in\R^{m_k}$, where $m_k\triangleq \rho_k-\rho_{k-1}$. A group may contain several outputs, or be empty when the rank does not increase at that stage. Unless specified otherwise, $\col_k$ and $\diag_k$ use increasing $k=1,\ldots,K$, omitting empty groups.

For the output-derivative map studied below, group $y_{[k]}$ contributes its derivatives of orders zero through $k-1$. The total number of scalar coordinates is therefore
\begin{equation*}
 d\triangleq \sum_{k=1}^K k m_k.
\end{equation*}
For the two-input system above, $K=2$, $\rho_1=1$, $\rho_2=2$, and $y_{[1]}=y_1$, $y_{[2]}=y_2$. The lower output derivatives forming these coordinates are $y_1,y_2,\dot y_2$, so $d=3$.

To continue elimination and differentiation, the algorithm also uses derivatives of each selected group from its selection order onward. For a smooth output signal $y$, define the derivative extraction operators
\begin{equation*}
\begin{aligned}
 \DD_\ell(y)&\triangleq\col\{y_{[j]}^{(r)}:1\le j\le r\le\ell\},\\
 \DD_k^-(y)&\triangleq\col\{y_{[j]}^{(r)}:1\le j<k,\ j\le r\le k\}.
\end{aligned}
\end{equation*}
Within each stack, groups are ordered by increasing $j$, and derivatives within a group by increasing $r$. The operator $\DD_k^-$ omits the block $y_{[k]}^{(k)}$ from $\DD_k$.

In the algebraic construction, $\theta_\ell$ denotes a vector of independent real parameters with the same dimension and ordering as $\DD_\ell(y)$. These parameters are called formal derivative variables. The vector $\theta_k^-$ is obtained from $\theta_k$ by omitting the block for group $k$. Both $\theta_0$ and $\theta_1^-$ are empty. Along a plant trajectory, the formulas are evaluated at $\theta_\ell=\DD_\ell(y)$ and $\theta_k^-=\DD_k^-(y)$. For a prescribed reference $r$, they are evaluated at $\DD_\ell(r)$ and $\DD_k^-(r)$. The rank conditions below hold for every value of the independent formal variables.

\subsection{Global inversion conditions}
We impose conditions that yield a smooth formula for the input throughout the state and formal derivative variable domains. The requirement below has the form of \cite[Assumption 1]{Liberzon2004} and is imposed here on the cumulative elimination procedure. Appendix~\ref{app:representation} specifies this procedure and its relation to the printed elimination formula in that reference.

\begin{assumption}\label{ass:G}
At every step $k$, the cumulative matrix consisting of the previously selected coefficient rows and the candidate coefficient rows obtained by differentiation has constant rank $\rho_k$. There is a fixed set of $\rho_k-\rho_{k-1}$ candidate rows that, together with the previously selected rows, forms a linearly independent set for every $x\in\R^n$ and every value of the formal derivative variables $\theta_{k-1}$.
\end{assumption}

We separately impose the full input rank condition used in \cite[Theorem 1]{Liberzon2004}.
\begin{assumption}\label{ass:F}
There is a finite integer $K$ such that $\rho_K=m$.
\end{assumption}

The following lemma records the derivative identities of modified inversion \cite{Singh1981,Liberzon2004}, using all retained rows for elimination as in \cite[Sec. 4]{IsidoriJSSC2025}. The identities at the selected derivative orders determine the input. Those at lower orders specify initial compatibility with a reference.

\begin{lemma}\label{lem:representation}
Under Assumptions~\ref{ass:G} and~\ref{ass:F}, there are globally smooth functions $a_k,Q_{k,\ell}$ taking values in $\R^{m_k}$ and $B_k$ taking values in $\R^{m_k\times m}$ such that, along every trajectory driven by a smooth input,
\begin{align}
 y_{[k]}^{(k)}
   &=a_k(x,\theta_k^-)+B_k(x,\theta_{k-1})u,
       &&1\le k\le K,\label{eq:terminal}\\
 y_{[k]}^{(\ell)}
   &=Q_{k,\ell}(x,\theta_\ell),
       &&0\le\ell<k.\label{eq:nonterminal}
\end{align}
Here $B_k$ is the coefficient matrix of $u$, and $a_k$ is the remaining term. At order zero, $Q_{k,0}=h_{[k]}$, where $h_{[k]}$ is the corresponding block of $h$. The matrix
\begin{equation*}
 \BB(x,\theta_{K-1})\triangleq \col_{k=1}^K B_k(x,\theta_{k-1})
\end{equation*}
is an $m\times m$ matrix that is nonsingular for every $x$ and every value of $\theta_{K-1}$. Each block $B_k$ uses the corresponding components $\theta_{k-1}$ of this same vector.
\end{lemma}

The proof is given in Appendix~\ref{app:representation}. Substituting the reference derivatives into \eqref{eq:terminal} and inverting $\BB$ gives a smooth input formula. Initial compatibility is defined from \eqref{eq:nonterminal} in Section~\ref{sec:main}.

\subsection{A uniform state bound}
For continuation of reference lifts, we retain the bound on the state over finite time intervals implied by weak uniform zero-state detectability \cite{LMS2002,Liberzon2004}. The following assumption bounds the state on each finite horizon in terms of the initial state and output derivatives, uniformly over inputs.
\begin{assumption}\label{ass:B}
There is a finite integer $N\ge0$ such that, for every time horizon $\tau>0$, initial state bound $R\ge0$, and output jet bound $M\ge0$, there is a finite number $\Omega(\tau,R,M)$ with the following property. Every trajectory of \eqref{eq:plant} driven by a smooth input satisfies, on its maximal existence interval,
\begin{equation*}
 |x(t)|\le\Omega(\tau,R,M)
\end{equation*}
whenever $0\le t\le\tau$, $|x(0)|\le R$, and $\|\jet{N}{y}\|_{[0,t]}\le M$. The bound is uniform over the inputs.
\end{assumption}

Along a reference lift, bounded reference derivatives and the smooth inversion equation allow this state bound to ensure continuation on each prescribed finite interval. The bound may grow with the time horizon. Weak uniform zero-state detectability of finite order, related to the output-to-state stability estimates in \cite{SW1997}, requires a finite integer $N\ge0$ and fixed functions $\beta\in\KL$, $\gamma\in\Kinf$ (enlarging the gain to class $\Kinf$ if necessary) such that, for every initial state and every smooth input, the corresponding trajectory satisfies
\begin{equation}\label{eq:WUD}
 |x(t)|\le\beta(|x(0)|,t)
       +\gamma\bigl(\|\jet{N}{y}\|_{[0,t]}\bigr),
\end{equation}
for every nonnegative time in its maximal existence interval. This estimate implies Assumption~\ref{ass:B} with
\begin{equation}\label{eq:omega}
 \Omega(\tau,R,M)\triangleq \beta(R,0)+\gamma(M).
\end{equation}

The converse does not hold. Assumption~\ref{ass:B} is strictly weaker than weak uniform zero-state detectability, as shown by the example in Section~\ref{sec:state-verification}. Under the same inversion conditions, Corollary~\ref{cor:prescribed-tracking} uses \eqref{eq:WUD} to bound both state and input for all time when the desired output and its required derivatives are uniformly bounded over time.

\section{Main Results}\label{sec:main}
\subsection{Compatibility and lifting}
Fix $\tau>0$ and an open interval $I$ containing $[0,\tau]$. A reference is a $C^\infty$ map $r:I\to\R^m$. For $x_0\in\R^n$, a lift of $(x_0,r)$ on $[0,\tau]$ is a smooth pair $(x,u)$ satisfying \eqref{eq:plant}, $x(0)=x_0$, and $h(x(t))=r(t)$.
\begin{definition}\label{def:compatibility}
A pair $(x_0,r)$ is compatible if, for all nonempty groups and all $0\le\ell<k$,
\begin{equation}\label{eq:compatibility}
 r_{[k]}^{(\ell)}(0)
 =Q_{k,\ell}\bigl(x_0,\DD_\ell(r)(0)\bigr).
\end{equation}
\end{definition}
The compatibility conditions match the prescribed initial output derivatives with the identities in Lemma~\ref{lem:representation}.

\begin{theorem}\label{thm:lifting}
Suppose Assumptions~\ref{ass:G}--\ref{ass:B} hold. Every compatible pair $(x_0,r)$ has a unique lift on $[0,\tau]$. It is generated by the globally defined feedback
\begin{equation}\label{eq:Ur}
\begin{aligned}
 U_r(t,x)\triangleq {}&\BB\bigl(x,\DD_{K-1}(r)(t)\bigr)^{-1}\\
 &{}\times\col_k\!\left[
 r_{[k]}^{(k)}(t)-a_k\bigl(x,\DD_k^-(r)(t)\bigr)
 \right].
\end{aligned}
\end{equation}
More precisely, the solution of
\begin{equation}\label{eq:liftODE}
 \dot x=f(x)+G(x)U_r(t,x),\qquad x(0)=x_0
\end{equation}
exists on $[0,\tau]$ and satisfies $h(x(t))=r(t)$. For every jointly smooth finite-dimensional family of compatible initial states and references, its endpoint depends smoothly on the parameters.
\end{theorem}

Uniqueness is for a fixed initial state. Different initial states may generate the same output.

\subsection{Surjectivity and global sections}
Fix homogeneous linear dynamics for each nonempty output group. For group $k$, choose constant matrices $L_{k,\ell}\in\R^{m_k\times m_k}$, $0\le\ell<k$, and prescribe
\begin{equation}\label{eq:assignment}
 y_{[k]}^{(k)}=\sum_{\ell=0}^{k-1}L_{k,\ell}y_{[k]}^{(\ell)}.
\end{equation}
Let $A_k\in\R^{km_k\times km_k}$ be the associated block companion matrix and $C_k\in\R^{m_k\times km_k}$ the selector of its first block. In $m_k\times m_k$ blocks, for $1\le i,j\le k$,
\[
 (A_k)_{ij}\triangleq \begin{cases}
 I_{m_k},&i<k,\ j=i+1,\\
 L_{k,j-1},&i=k,\\
 0,&\text{otherwise}.
 \end{cases}
\]
The selector is $C_k\triangleq \begin{pmatrix}I_{m_k}&0&\cdots&0\end{pmatrix}$. Set
\begin{equation}\label{eq:AC}
 A\triangleq \diag_k A_k,\qquad C\triangleq \diag_k C_k.
\end{equation}
Hurwitz stability of the companion matrices is not required for the geometric conclusions.

Under Assumptions~\ref{ass:G} and~\ref{ass:F}, Lemma~\ref{lem:rigidity} constructs a global smooth feedback $\alpha:\R^n\to\R^m$ realizing \eqref{eq:assignment}. With $F\triangleq f+G\alpha$, use $L_F^0h\triangleq h$ and $L_F^{\ell+1}h\triangleq D(L_F^\ell h)F$. Its associated output-derivative map $T:\R^n\to\R^d$ is
\begin{equation}\label{eq:Tdef}
 T(x)\triangleq \col_{k=1}^K T_{[k]}(x),\qquad
 T_{[k]}\triangleq \col_{0\le\ell<k}L_F^\ell h_{[k]}.
\end{equation}
Lemma~\ref{lem:assignment-unique} shows that every pair of globally smooth maps realizing the same assignment coincides with $(\alpha,T)$.

\begin{theorem}\label{thm:section}
Suppose Assumptions~\ref{ass:G} and~\ref{ass:F} hold. For every choice of
\eqref{eq:assignment}, the associated map $T$ satisfies
$\operatorname{rank}DT(x)=d$ for all $x\in\R^n$. In particular,
$d\le n$.

If Assumption~\ref{ass:B} also holds, then for every $\tau>0$
there is a smooth global section
\begin{equation*}
 R_\tau:\R^d\longrightarrow\R^n,
 \qquad T\circ R_\tau=\id_{\R^d}.
\end{equation*}
The map $T$ is then a surjective submersion, and $R_\tau$ is a
proper smooth embedding. If $d=n$, then $T$ is a global
diffeomorphism. The section constructed below satisfies
$R_\tau(0)=0$ whenever $T(0)=0$ and $F(0)=0$.

If \eqref{eq:WUD} holds, then $T(0)=F(0)=0$, the constructed section satisfies $R_\tau(0)=0$, and
\begin{equation}\label{eq:sectionbound}
 |R_\tau(b)|\le\gamma(c_\tau|b|)
\end{equation}
for all $b\in\R^d$, where the finite constant $c_\tau>0$ depends on the assignment, $N$, and the chosen reference family.
\end{theorem}

The section is constructed by prescribing terminal output derivatives and taking the endpoint of a reference lift. For each target $b\in\R^d$, Section~\ref{sec:geometry} constructs a compatible reference $r_b$ with these terminal derivatives. Theorem~\ref{thm:lifting} gives its state lift $x_b$ from the origin through time $\tau$. The terminal identity established in Section~\ref{sec:geometry} yields
\begin{equation*}
 R_\tau(b)=x_b(\tau),\qquad T(R_\tau(b))=b.
\end{equation*}
The full rank conclusion follows separately from local lifting under Assumptions~\ref{ass:G} and~\ref{ass:F}.

\begin{corollary}\label{cor:Liberzon}
Suppose Assumptions~\ref{ass:G} and~\ref{ass:F} hold and the system is weakly uniformly zero-state detectable. Fix the selected output groups and homogeneous assignment. Let $\alpha_L:\R^n\to\R^m$ and $T_L:\R^n\to\R^d$ be global smooth maps realizing that assignment in the sense of \cite[Sec. 6]{Liberzon2004}, namely, with $F_L\triangleq f+G\alpha_L$, they satisfy $DT_LF_L=AT_L$ and $h=CT_L$. Then $\alpha_L=\alpha$, $T_L=T$, and $T_L$ admits the section in Theorem~\ref{thm:section}. In particular, $T_L$ is onto.
\end{corollary}
\begin{remark}\label{rem:surjectivity-significance}
Under Assumptions~\ref{ass:G} and~\ref{ass:F}, Corollary~\ref{cor:Liberzon} makes the separate onto requirement in \cite[Th. 2]{Liberzon2004} a consequence of weak uniform zero-state detectability. For the fixed output groups and assignment in Corollary~\ref{cor:Liberzon}, these conditions answer the surjectivity question in \cite[slide 119]{Isidori2017}. Theorem~\ref{thm:section} further establishes surjectivity and a smooth global section under the strictly weaker state bound in Assumption~\ref{ass:B}.
\end{remark}

\subsection{Tracking and coordinate realization}
Complete lifting also gives a controller that achieves exact output tracking after a prescribed time from
any initial state. Choose the homogeneous assignment
so that $A$ in \eqref{eq:AC} is Hurwitz, and let $T$ be its associated map.
Given $\tau>0$, choose a
smooth $\chi:\R\to[0,1]$ that is zero near $0$ and one near and after $\tau$.
Let the desired output $y_d:I\to\R^m$ be smooth on an open interval $I$
containing $[0,\infty)$. Introduce a controller state $\zeta\in\R^d$ and use
\begin{equation}\label{eq:tracking-generator}
\begin{aligned}
 \dot\zeta&=A\zeta,\qquad \zeta(0)=T(x_0),\\
 r(t)&\triangleq (1-\chi(t))C\zeta(t)+\chi(t)y_d(t),\\
 u(t)&=U_r(t,x(t)).
\end{aligned}
\end{equation}
Here $U_r$ is given by \eqref{eq:Ur}. The derivatives of $r$ are computed
from the generator equation and the current derivatives of $y_d$ through
order $K$. Thus the controller uses the full plant state and current
reference data, with initialization from $x_0$.

\begin{corollary}\label{cor:prescribed-tracking}
Under Assumptions~\ref{ass:G}--\ref{ass:B}, the controller
\eqref{eq:tracking-generator} gives a smooth solution on $[0,\infty)$ for
every $x_0\in\R^n$, and
\begin{equation}\label{eq:prescribed-tracking}
 h(x(t))=y_d(t),\qquad t\ge\tau.
\end{equation}
If \eqref{eq:WUD} holds with order $N$, set $\nu\triangleq \max\{K,N\}$. For every $R,M\ge0$,
the plant state, controller state, and input have a common finite bound
over all $t\ge0$, all initial states with $|x_0|\le R$, and all references
with $\sup_{t\ge0}|j^\nu y_d(t)|\le M$. The bound may depend on $R,M,\nu$,
the plant, the assignment, $\tau$, and the chosen $\chi$.
\end{corollary}

The desired output need not be compatible with $x_0$ because the generator
supplies a compatible initial segment. The same smooth control formula
remains defined through and after $\tau$. Section~\ref{sec:tracking-proof}
proves the result and gives a state estimate after tracking is attained.

For any homogeneous assignment, complete lifting also realizes a coordinate target from the actual initial state and joins the nominal feedback smoothly.
\begin{corollary}\label{cor:coordinate-maneuver}
Suppose Assumptions~\ref{ass:G}--\ref{ass:B} hold. For each homogeneous assignment and each $\tau>0$, there is a smooth map $S_\tau:\R^n\times\R^d\to\R^n$ such that every initial state $x_0$ and target $b$ admit a smooth input whose trajectory exists on $[0,\infty)$ and satisfies
\begin{equation*}
 x(0)=x_0,\quad x(\tau)=S_\tau(x_0,b),\quad
 T(S_\tau(x_0,b))=b.
\end{equation*}
The input agrees with $\alpha(x)$ near $\tau$ and thereafter, and $T(x(t))=e^{A(t-\tau)}b$ for $t\ge\tau$. If \eqref{eq:WUD} holds and $A$ is Hurwitz, then $x(t)\to0$ and $u(t)\to\alpha(0)$, with $\alpha(0)=0$ whenever $f(0)=0$.
\end{corollary}
The map $T$ gives the assigned derivative coordinates. The section $R_\tau$ is constructed from the origin, whereas $S_\tau$ allows the initial state to vary. The target $b$ specifies the assigned output derivatives at time $\tau$. When $d<n$, the terminal state $x(\tau)$ is selected by the lift and depends on $x_0$. Section~\ref{sec:coordinate-proof} proves this result.

\section{Complete Output Lifting}\label{sec:liftproof}
This section proves Theorem~\ref{thm:lifting}.
\subsection{Local tracking}
Lemma~\ref{lem:representation} makes \eqref{eq:Ur} smooth in $(t,x)$ on a neighborhood of $[0,\tau]\times\R^n$, and hence \eqref{eq:liftODE} has a unique maximal local solution \cite[Ths. 2.2 and 2.13]{Teschl2012}.

Proceed by induction through the nonempty groups. Suppose all groups preceding stage $k$ track their references. All derivatives of these groups entering the selected input equation for group $k$ therefore agree with the corresponding reference derivatives, so \eqref{eq:terminal} and \eqref{eq:Ur} give
\begin{equation}\label{eq:firsttracking}
 y_{[k]}^{(k)}=r_{[k]}^{(k)}.
\end{equation}
For each $\ell<k$, \eqref{eq:nonterminal} gives the actual initial derivative as $Q_{k,\ell}$ evaluated at $x_0$ and the earlier actual jets. These jets agree with the reference, and \eqref{eq:compatibility} supplies the matching initial value. Integrating \eqref{eq:firsttracking} with these matching initial derivatives gives $y_{[k]}=r_{[k]}$. For the first nonempty group, the earlier jet arguments are empty, so the same step starts the induction. The induction proves
\begin{equation*}
 h(x(t))=r(t)
\end{equation*}
throughout the maximal interval in $[0,\tau]$.

\subsection{Continuation and smooth dependence}\label{sec:liftdependence}
\begin{lemma}\label{lem:continuation}
Assume the state estimate in Assumption~\ref{ass:B}. Let $H(t,x)$ take values in $\R^m$ and be smooth for all $x\in\R^n$ and for $t$ in a neighborhood of $[0,\tau]$. Consider the maximal solution of $\dot x=f(x)+G(x)H(t,x)$ with $x(0)=x_0$. Suppose its output equals a prescribed smooth reference $r$ throughout its interval of existence in $[0,\tau]$. Then the solution reaches $\tau$ and, for every $0\le t\le\tau$,
\begin{equation}\label{eq:liftbound}
 |x(t)|\le\Omega(\tau,|x_0|,M_r),\qquad
 M_r\triangleq \max_{0\le s\le\tau}|\jet{N}{r}(s)|.
\end{equation}
For a family with uniformly bounded initial states and reference $N$-jets, a common state bound holds on $[0,\tau]$.
\end{lemma}
\begin{proof}
Let $T_{\max}$ be the maximal forward existence time. For every $t_0<\min\{T_{\max},\tau\}$, a cutoff extends the realized input $u(t)=H(t,x(t))$ from $[0,t_0]$ to a globally defined smooth open-loop input. Uniqueness identifies its plant trajectory with the lift on $[0,t_0]$. Assumption~\ref{ass:B} therefore gives \eqref{eq:liftbound} throughout the existing interval, with the same bound independent of the cutoff.

If $T_{\max}\le\tau$, this estimate confines the trajectory to a compact ball. The closed-loop vector field is smooth on a neighborhood of that ball times $[0,\tau]$, so the continuation criterion in \cite[Cor. 2.15]{Teschl2012} extends the solution through $T_{\max}$, contradicting maximality. Continuity gives \eqref{eq:liftbound} at $\tau$. Applying Assumption~\ref{ass:B} with common bounds on initial states and reference jets gives the uniform assertion.
\end{proof}

Applying Lemma~\ref{lem:continuation} with $H=U_r$ completes the existence proof in Theorem~\ref{thm:lifting}. Once the state is confined to a compact set, boundedness of the reference derivatives on $[0,\tau]$ places all arguments of $\BB^{-1}$ in a compact set, on which its smoothness ensures boundedness. No growth condition on $\BB^{-1}$ at infinity is required.

If another smooth trajectory has the same initial state and output $r$, the derivative identities \eqref{eq:terminal} recover the same input \eqref{eq:Ur}. It therefore solves the same initial value problem, whose solution is unique. This proves uniqueness of the lift.

Let $\lambda$ vary in an open subset of $\R^a$, with jointly smooth compatible data $x_0(\lambda)$ and $r(t,\lambda)$. For each fixed $\lambda_0$, the lift exists through $\tau$ by Lemma~\ref{lem:continuation}. The lifting equation is jointly smooth in $(t,x,\lambda)$. By the smooth flow theorem, with time and parameters adjoined as states, the set of initial states and parameters whose solutions exist through $\tau$ is open, and the endpoint depends smoothly on these data \cite[Th. 6.1]{Teschl2012}. Evaluating this endpoint map at $(x_0(\lambda),\lambda)$ gives smooth dependence on $\lambda$ near $\lambda_0$. Since $\lambda_0$ is arbitrary, this proves the parameter statement in Theorem~\ref{thm:lifting}.

The input formula \eqref{eq:Ur} depends on the reference through its $K$-jet. We write it as $U_r(t,x)=\mathcal U(x,j^Kr(t))$, where $\mathcal U:\R^n\times\R^{m(K+1)}\to\R^m$ is the smooth inversion map. Its second argument contains all reference derivatives from order zero through $K$. Appendix~\ref{app:quantitative} gives quantitative estimates and the variational equation for parameter sensitivities.

\section{Surjectivity and Global Geometry}\label{sec:geometry}
Subsections A--C prove Theorem~\ref{thm:section} and Corollary~\ref{cor:Liberzon}. Subsection D gives an alternative sufficient condition for a global section.
\subsection{Homogeneous assignment and uniqueness}
\begin{lemma}\label{lem:rigidity}
Under Assumptions~\ref{ass:G} and~\ref{ass:F}, for the assignment \eqref{eq:assignment}, there exist global smooth maps $T$ and $\alpha$ satisfying
\begin{equation}\label{eq:semiconj}
 DT(x)F(x)=AT(x),\qquad h(x)=CT(x).
\end{equation}
Identity~\eqref{eq:Tdef} holds. If $f(0)=0$ and $h(0)=0$, then $T(0)=0$, $\alpha(0)=0$, and $F(0)=0$. If any smooth plant trajectory satisfies all equations \eqref{eq:assignment} on an interval, the stack of its selected output derivatives in \eqref{eq:Tdef} equals $T(x)$ there and its input equals $\alpha(x)$.
\end{lemma}

The construction substitutes the assigned homogeneous derivatives into the identities of Lemma~\ref{lem:representation}, successively through the output groups, and then solves the selected input equations for $\alpha$. Appendix~\ref{app:assignment} gives the full proof.

Lemma~\ref{lem:rigidity} identifies the derivative stack of any lifted trajectory once its output satisfies the assigned homogeneous equations. It therefore applies on the terminal interval used to construct the section. The same argument identifies the map produced by any feedback realizing the chosen assignment.
\begin{lemma}\label{lem:assignment-unique}
Suppose Assumptions~\ref{ass:G} and~\ref{ass:F} hold, and fix the output groups and the companion pair $(A,C)$ of \eqref{eq:assignment}. Let global smooth maps $\bar\alpha:\R^n\to\R^m$ and $\bar T:\R^n\to\R^d$ satisfy
\begin{equation*}
 \bar F\triangleq f+G\bar\alpha,\qquad
 D\bar T\,\bar F=A\bar T,\qquad h=C\bar T.
\end{equation*}
Then $\bar\alpha=\alpha$ and $\bar T=T$ on $\R^n$.
\end{lemma}
\begin{proof}
Write $\bar T=\col_k\bar T_{[k]}$, where $\bar T_{[k]}:\R^n\to\R^{km_k}$. Along a local trajectory of $\bar F$ through any state $x_0$, the vector $\zeta\triangleq \bar T(x)$ satisfies $\dot\zeta=A\zeta$ and $y=C\zeta$. For every nonempty group, the selector matrices satisfy
\[
 \col_{0\le\ell<k}(C_kA_k^\ell)=I_{km_k}.
\]
Thus the actual output derivatives of orders zero through $k-1$ for that group equal $\bar T_{[k]}(x)$, and its derivative of order $k$ satisfies \eqref{eq:assignment}. Lemma~\ref{lem:rigidity} identifies these same derivatives with $T_{[k]}(x)$ and the input with $\alpha(x)$. Evaluating at $t=0$ gives both claimed equalities at the arbitrary state $x_0$.
\end{proof}

\subsection{Terminal references and differential rank}
Under Assumptions~\ref{ass:G} and~\ref{ass:F}, for each $\tau>0$ choose a smooth function $\chi$ that is zero near $0$ and one near $\tau$. For an initial state $x_0$ and a target $b=\col_{k=1}^K b_{[k]}\in\R^d$, with $b_{[k]}\in\R^{km_k}$, define
\begin{equation}\label{eq:rb}
 r_{x_0,b}(t)\triangleq (1-\chi(t))Ce^{At}T(x_0)+\chi(t)Ce^{A(t-\tau)}b.
\end{equation}
Write $r_b\triangleq r_{0,b}$ for the references from the origin. By \eqref{eq:semiconj}, $r_{x_0,b}$ agrees near $0$ with the output of the local nominal trajectory from $x_0$, so it is compatible with $x_0$. Near $\tau$, it satisfies \eqref{eq:assignment} with terminal derivative stack $b$.

To prove full rank at an arbitrary state $x\in\R^n$, we construct a local right inverse of $T$ whose value at $T(x)$ is $x$. Let $\varphi_t$ be the local flow of $F=f+G\alpha$, and choose $\tau>0$ so small that $\varphi_t(x)$ is defined on a neighborhood of $[-\tau,0]$. Set $x_-\triangleq \varphi_{-\tau}(x)$, so the nominal trajectory $\varphi_{t-\tau}(x)$ from $x_-$ reaches $x$ at time $\tau$. Since $T(x_-)=e^{-A\tau}T(x)$, the reference $r_{x_-,b}$ with $b=T(x)$ equals $Ce^{A(t-\tau)}T(x)$ on $[0,\tau]$, and its lift is this nominal trajectory. The lifting field is jointly smooth in $(t,x,b)$. By the same smooth flow theorem, the lifts from $x_-$ also reach $\tau$ for $b$ near $T(x)$. Denote their states at time $\tau$ by $E(b)$, which depends smoothly on $b$. Local tracking and Lemma~\ref{lem:rigidity} yield
\begin{equation*}
 T(E(b))=b,\qquad E(T(x))=x.
\end{equation*}
Differentiating at $b=T(x)$ gives
\[
 DT(x)DE(T(x))=I_d.
\]
Thus $\operatorname{rank}DT(x)=d$ at every state, and $d\le n$.

\subsection{Global sections and geometric consequences}
Assume now that Assumption~\ref{ass:B} also holds. For each $\tau>0$ and $b\in\R^d$, Theorem~\ref{thm:lifting} gives the complete lift $x_b$ of $r_b$ from the origin. Define
\begin{equation*}
 R_\tau(b)\triangleq x_b(\tau).
\end{equation*}
Near the terminal time, the reference satisfies all homogeneous equations \eqref{eq:assignment}. The lift tracks this reference, so Lemma~\ref{lem:rigidity} identifies its selected actual output derivatives with $T(x_b)$. At $t=\tau$,
\begin{align}\label{eq:endpointidentity}
 T(x_b(\tau))
 &=\col_{k,\,0\le\ell<k} r_{b,[k]}^{(\ell)}(\tau)\\
 &=\col_{k,\,0\le\ell<k} C_kA_k^\ell b_{[k]}=b.\nonumber
\end{align}
The last equality follows because $C_kA_k^\ell$ extracts the $\ell$th derivative block for $0\le\ell<k$. Equation~\eqref{eq:endpointidentity} gives the required section identity for every $b\in\R^d$.

Theorem~\ref{thm:lifting} gives smoothness of $R_\tau$. If $T(0)=0$ and $F(0)=0$, then $r_b(t)=\chi(t)Ce^{A(t-\tau)}b$, and the zero target has the zero state lift. Thus $R_\tau(0)=0$.

Differentiating the endpoint identity yields
\begin{equation*}
 DT(R_\tau(b))\,D_bR_\tau(b)=I_d.
\end{equation*}
From $T\circ R_\tau=\id_{\R^d}$, the section is an injective immersion with inverse $T$ on its image, hence a smooth embedding. The identity
\begin{equation*}
 \im R_\tau=\{x\in\R^n:R_\tau(T(x))=x\}
\end{equation*}
shows that the image is closed. For any compact set $H\subset\R^n$, the identity $R_\tau^{-1}(H)=T(H\cap\im R_\tau)$ shows that its inverse image is compact. Thus $R_\tau$ is proper.

When $n=d$, differentiating $T\circ R_\tau=\id_{\R^d}$ shows that $DR_\tau$ is invertible. The inverse function theorem therefore makes $\im R_\tau$ open. Since the image is nonempty and $\R^n$ is connected, $\im R_\tau=\R^n$, and $T,R_\tau$ are mutual smooth inverses.

Under \eqref{eq:WUD}, the section identity itself implies $T(0)=0$ and $F(0)=0$. Choose $\xi$ with $T(\xi)=0$. Its nominal trajectory has zero output, so \eqref{eq:WUD} gives $|x(t)|\le\beta(|\xi|,t)$ on the existing interval. Lemma~\ref{lem:continuation} gives completeness, and hence $x(t)\to0$. Since $T(x(t))=0$, continuity yields $T(0)=0$. Starting instead at the origin, the same zero-output estimate gives $x(t)=0$, so $F(0)=0$. The reference $r_b$ is therefore linear in $b$, and $\|\jet{N}{r_b}\|_{[0,\tau]}\le c_\tau|b|$ for a finite constant depending on $\tau,A,N$ and the chosen cutoff $\chi$. Applying \eqref{eq:WUD} from the origin proves \eqref{eq:sectionbound}.

This completes the proof of Theorem~\ref{thm:section} and, by \eqref{eq:omega}, Corollary~\ref{cor:Liberzon}.

For $d<n$, $T$ need not be injective or proper. A coordinate projection with an unconstrained internal state has a proper smooth section but is itself not proper.

Global normal forms provide another route to global tracking and feedback design \cite{Schwartz1999,Isidori1999,RAO2006}. Within the assumed global normal form, triangular dependence of the multiplier functions and trivial zero dynamics permit the explicit global diffeomorphism in \cite[Th. 1, proof of (32)]{Wang2017}. The endpoint construction here also permits nontrivial internal dynamics. When $d<n$, it supplies a section of $T$ without requiring a global parametrization of the fibers.

\begin{corollary}\label{cor:escape-certificate}
Assume only Assumptions~\ref{ass:G} and~\ref{ass:F}, and construct $T$ from a fixed homogeneous assignment. If $b\notin T(\R^n)$, then the reference $r_b$ in \eqref{eq:rb} has a locally exact lift from the origin that cannot be continued to $\tau$. Along its maximal interval every finite output jet is bounded, while the state leaves every compact subset of $\R^n$.
\end{corollary}

This corollary identifies the obstruction behind the onto condition. Under Assumptions~\ref{ass:G} and~\ref{ass:F}, a target outside the image of $T$ forces its reference lift to escape in finite time. Assumption~\ref{ass:B} removes precisely this possibility.

\begin{proof} The local tracking part of Theorem~\ref{thm:lifting} uses only Assumptions~\ref{ass:G} and~\ref{ass:F} because \eqref{eq:Ur} is smooth and induction over the selected groups gives $h(x_b)=r_b$ for as long as the solution exists. If the lift reached $\tau$, Lemma~\ref{lem:rigidity} would identify $T(x_b(\tau))$ with the terminal derivative stack of $r_b$, which equals $b$ by the identities for the selector matrices, a contradiction. Hence its maximal time satisfies $T_{\max}\le\tau$. Since $r_b$ is smooth on the fixed interval, all finite jets of the output remain bounded. The lifting vector field is smooth on every compact state set times this bounded jet set. Therefore, the maximal forward existence time can be finite only if the state leaves every compact subset of $\R^n$. \end{proof}

\subsection{A sufficient condition for a global section}
Surjectivity can also be established by bounding only the reference lifts used to construct the section. The following criterion applies even when Assumption~\ref{ass:B} fails.
\begin{proposition}
\label{prop:terminal-family-certificate}
Suppose Assumptions~\ref{ass:G} and~\ref{ass:F} hold. Fix $\tau>0$, a homogeneous
assignment, and the references $r_b$ in \eqref{eq:rb}. Write
$\mathcal V_b(t,x)\triangleq f(x)+G(x)U_{r_b}(t,x)$.
Suppose an embedded submanifold $S\subset\R^n$ without
boundary contains $0$, and $\mathcal V_b(t,x)\in T_xS$ for all
$x\in S$, $t\in[0,\tau]$, and $b\in\R^d$, where $T_xS$ denotes the tangent space to $S$ at $x$.
Suppose that, for every $b\in\R^d$, there are a proper $C^1$ function
$V_b:S\to[0,\infty)$, meaning that each sublevel set is compact in $S$, and finite constants
$a_b,c_b\ge0$ satisfying
\begin{equation}\label{eq:terminal-family-certificate}
 DV_b(x)\mathcal V_b(t,x)\le a_bV_b(x)+c_b,\quad x\in S,\ t\in[0,\tau].
\end{equation}
Here $DV_b(x)$ denotes the differential of $V_b$ on $S$.
Then $R_\tau(b)\triangleq x_b(\tau)$, obtained by lifting $r_b$ from $0$,
is a smooth global section of $T$ and a proper smooth embedding.
If $d=n$, then $T$ is a global diffeomorphism. If $T(0)=0$ and $F(0)=0$,
then $R_\tau(0)=0$.
\end{proposition}
\begin{proof}
Fix $b$. Tangency defines a smooth vector field on $S$. Let $x_b$ be its maximal solution from $0$, which uniqueness identifies with the ambient lift. Compatibility of $r_b$ with the zero initial state gives local tracking under Assumptions~\ref{ass:G} and~\ref{ass:F}. On the existing interval with $t\le\tau$, \eqref{eq:terminal-family-certificate} and Gronwall's inequality give
\[
 V_b(x_b(t))\le e^{a_b\tau}\bigl(V_b(0)+c_b\tau\bigr).
\]
Properness of $V_b$ confines this solution to a compact subset of $S$. Smoothness of the restricted vector field therefore permits continuation to $\tau$.

For each fixed $b_0$, the solution exists through $\tau$, and the ambient lifting field is jointly smooth in $(t,x,b)$. Smooth dependence on parameters therefore makes the endpoint smooth near $b_0$. Since $b_0$ is arbitrary, the endpoint map is smooth on $\R^d$. Lemma~\ref{lem:rigidity} gives $T(R_\tau(b))=b$. If $T(0)=0$ and $F(0)=0$, the zero target has the zero reference and zero state lift, giving $R_\tau(0)=0$. The preceding embedding and properness arguments and the argument for equal dimensions apply to this section identity.
\end{proof}

This criterion tests only the reference lifting fields on
$S$. For $\dot y=u$, $\dot z=z^2$, with output $y$, take
$S\triangleq \{z=0\}$ and $V_b\triangleq y^2$ for every $b$. Along the reference lifting field,
$\dot V_b=2y\dot r_b\le V_b+|\dot r_b|^2$, so the criterion
holds and $T(y,z)=y$ is onto. Nevertheless, $u=0$, $y(0)=0$,
and $z(0)>0$ give finite escape with zero output, violating
Assumption~\ref{ass:B}.

\section{Consequences for Feedback Design}\label{sec:consequences}
We prove Corollaries~\ref{cor:prescribed-tracking} and~\ref{cor:coordinate-maneuver}, then apply the surjectivity result to output regulation.
\subsection{Tracking and state bounds}\label{sec:tracking-proof}
Existing tracking methods include inversion through singular points for SISO systems \cite{HD1988} and static feedback based on dynamic linearizability \cite{RAO2006}. Prescribed-time tracking has also been established for uncertain MIMO systems in strict-feedback-like form \cite{YeSong2023} and, through funnel control, for unknown nonlinear systems with constraints on the transient error \cite{Lanza2024}.

Corollary~\ref{cor:prescribed-tracking} derives its global guarantee from complete lifting under Assumptions~\ref{ass:G}--\ref{ass:B}. For a known plant, the controller supplies a compatible initial reference segment and smoothly joins any prescribed smooth desired output. The resulting solution exists for all time and tracks that output exactly from the prescribed time onward. The construction accommodates internal dynamics and the original outputs without a uniform vector relative degree in Section~\ref{sec:skew}. Under weak uniform zero-state detectability, it also gives uniform bounds on the plant state, controller state, and input for bounded initial states and reference derivatives as stated in the corollary.

\begin{proof}[Proof of Corollary~\ref{cor:prescribed-tracking}]
The generator gives $\zeta(t)=e^{At}T(x_0)$. Near $0$, its output agrees
with the local nominal plant output by \eqref{eq:semiconj}. Hence $r$ is
compatible with $x_0$. Theorem~\ref{thm:lifting} gives its complete lift on
every finite interval. Uniqueness makes these lifts agree, producing one
smooth solution on $[0,\infty)$. Since $\chi=1$ near and after $\tau$,
the identity $h(x)=r$ gives \eqref{eq:prescribed-tracking}.
Differentiating $r$ using $\dot\zeta=A\zeta$ shows that its $K$-jet depends
only on $t,\zeta$ and the current $K$-jet of $y_d$.

For bounded data, set $T_R\triangleq \max_{|x|\le R}|T(x)|$. Leibniz's rule, the fixed
cutoff, and the Hurwitz property give a finite constant $c>0$, depending on
$A,C,\nu$ and $\chi$, such that
\[
 \sup_{t\ge0}|j^\nu r(t)|\le c(T_R+M).
\]
Equation~\eqref{eq:WUD} therefore bounds the state by
$\beta(R,0)+\gamma(c(T_R+M))$ for all time. The reference $K$-jet lies in
a fixed closed ball. The smooth inversion map $\mathcal U$ defined in
Section~\ref{sec:liftdependence} consequently bounds the input on the product of the
state and jet balls. Finally, $e^{At}T(x_0)$ is uniformly bounded by the
Hurwitz property. These bounds are uniform over the stated data sets.
\end{proof}

After tracking is attained, applying \eqref{eq:WUD} with initial time
$\tau$ yields
\begin{equation*}
 |x(t)|\le\beta(|x(\tau)|,t-\tau)
 +\gamma\!\left(\sup_{\tau\le s\le t}|j^N y_d(s)|\right),
 \quad t\ge\tau.
\end{equation*}
For $y_d\equiv0$, the output vanishes by $\tau$ and remains zero, while
$|x(t)|\le\beta(|x(\tau)|,t-\tau)$. Lemma~\ref{lem:rigidity} then gives
$u(t)=\alpha(x(t))$ for $t\ge\tau$, so $u(t)\to\alpha(0)$. The internal
state converges asymptotically, and exact tracking does not require it to vanish
at $\tau$. The construction uses exact initialization and reference
derivatives. Its input bound depends on the chosen time and supplies no
fixed actuator limit as $\tau$ decreases.

\subsection{Stabilization and coordinate realization}\label{sec:coordinate-proof}
We first recall the detectability argument \cite{SW1997,LMS2002,Liberzon2004} used for the convergence assertion in Corollary~\ref{cor:coordinate-maneuver}.
\begin{lemma}\label{lem:stability}
Suppose \eqref{eq:WUD} and \eqref{eq:semiconj} hold, and $A$ is Hurwitz. Then $T(0)=0$, and the origin of $\dot x=F(x)$ is globally asymptotically stable.
\end{lemma}
\begin{proof}
By \eqref{eq:semiconj} and the Hurwitz property, there are $L,a>0$ such that $|\jet{N}{y}(t)|\le Le^{-at}|T(x_0)|$. Set
\begin{equation*}
 M_0\triangleq \beta(|x_0|,0)+\gamma(L|T(x_0)|).
\end{equation*}
The input extension in Lemma~\ref{lem:continuation} permits applying \eqref{eq:WUD} on the existing interval, giving $|x(t)|\le M_0$. Smoothness of $F$ then gives completeness. Apply \eqref{eq:WUD} on $[t/2,t]$ to obtain
\begin{equation*}
 |x(t)|\le\beta(M_0,t/2)
       +\gamma\bigl(Le^{-at/2}|T(x_0)|\bigr).
\end{equation*}
The right side tends to zero uniformly on bounded sets of initial states. Since $x(t)\to0$ and $T(x(t))=e^{At}T(x_0)\to0$, continuity gives $T(0)=0$. The nominal trajectory from the origin then has zero output, and \eqref{eq:WUD} gives $x(t)=0$, so $F(0)=0$. Finally, $M_0\to0$ as $x_0\to0$, and the bound $|x(t)|\le M_0$ proves Lyapunov stability.
\end{proof}
\begin{proof}[Proof of Corollary~\ref{cor:coordinate-maneuver}]
Use $r_{x_0,b}$ from \eqref{eq:rb}, extending $\chi$ to be one after $\tau$. The reference is compatible with $x_0$. Theorem~\ref{thm:lifting} gives its lift on every finite interval. Uniqueness makes these lifts agree. Smooth parameter dependence defines $S_\tau(x_0,b)\triangleq x(\tau)$. Near and after $\tau$, the reference satisfies the assigned homogeneous equations. Lemma~\ref{lem:rigidity} and the identities for the selector matrices give $T(x(t))=e^{A(t-\tau)}b$ and $u(t)=\alpha(x(t))$ there.

Under \eqref{eq:WUD} and the Hurwitz condition, Lemma~\ref{lem:stability} gives state convergence, and continuity yields $u(t)\to\alpha(0)$. If $f(0)=0$, then $h(0)=CT(0)=0$, and Lemma~\ref{lem:rigidity} gives $\alpha(0)=0$.
\end{proof}
For example, $\dot x=1+u$, $y=x$ satisfies \eqref{eq:WUD} with $N=0$. Assigning $\dot y=-y$ gives $\alpha(x)=-1-x$, so the state converges to zero while the input converges to $-1$.

\subsection{Application to multivariable output regulation}\label{sec:2026}
Under the global inversion and detectability hypotheses stated below, the surjectivity requirement in the output regulation result of \cite[Th. 6.1]{Isidori2026} follows from complete lifting. This application requires identifying the corrected coordinates used in regulation designs \cite{Wu2020,Wu2021,IsidoriWu2023} with terminal reference derivatives. We establish this identification for the three-output construction in \cite[Sec. 5]{Isidori2026}, along an exosystem trajectory contained in a compact invariant set.

Write the source's state as $x=\col(w,\widetilde\eta,\widetilde z,X)$ and its plant error state as $q\triangleq \col(\widetilde z,X)$. Here $w$ is the exosystem state, $\widetilde\eta$ is the internal model error, and $\widetilde z$ is the plant internal error coordinate. Let $\bar u$ denote its auxiliary control and $\Gamma(w,\widetilde\eta)$ the shifted output map of its internal model. The effective input in \cite[eq. (19)]{Isidori2026} is $v\triangleq \bar u+\Gamma(w,\widetilde\eta)$. With this input, the error equations become
\begin{align}
 \dot q&=f_e(w,q)+G_e(w,q)v,\qquad e=X_1,\label{eq:error-effective}\\
 \dot{\widetilde\eta}&=M_\eta\widetilde\eta+N_\eta v.\notag
\end{align}
Here $X=\col(X_1,X_2,X_3)$ consists of blocks of dimensions $3,2,1$, and $e=\col(e_1,e_2,e_3)$. The matrices $M_\eta,N_\eta$ are the constant matrices of the source's internal model. The nonlinear coefficients depend on $(w,q)$, independently of $\widetilde\eta$, on the global domain of the error states, with the smoothness required by the displayed derivatives. The zero-error state is invariant for $v=0$. Fix a trajectory of $\dot w=s(w)$ in the prescribed compact invariant set $W$ and a terminal time $\tau>0$.

Let $K_2=\col(K_2^{(2)},K_2^{(3)})\in\R^2$ and $K_3\in\R$ be the source's correction functions. The superscripts $(2)$ and $(3)$ identify the scalar corrections for outputs $e_2$ and $e_3$, not derivative orders. They satisfy
\begin{equation}\label{eq:correctednonterminal}
 \col(e_2',e_3')=X_2-K_2,\qquad e_3''=X_3-K_3.
\end{equation}
The functions depend on the state and formal output derivatives. Let $b_1,b_2,b_3$ denote the coefficient rows of $v$ in the selected equations for $e_1',e_2'',e_3'''$, and put $B\triangleq \col(b_1,b_2,b_3)$. Appendix~\ref{app:regulation-proof} gives the correction functions and the inverse input formula.

Fix the source's homogeneous assignment and denote its corrected map \cite[eq. (29)]{Isidori2026} by $\mathcal T:W\times\R^{n_e}\to\R^6$, where $n_e$ is the dimension of $\col(\widetilde\eta,q)$. Its three blocks $\zeta_1,\zeta_2,\zeta_3$ have dimensions $1,2,3$. The detectability condition in \cite[Sec. 5.3]{Isidori2026} bounds $q$ through $e$. Namely, for some $\beta\in\KL$, $\gamma\in\mathcal K$, and integer $N\ge0$, every initial state $q(0)$ and every smooth effective input give
\begin{equation}\label{eq:2026qbound}
 |q(t)|\le\beta(|q(0)|,t)+\gamma(\|\jet{N}{e}\|_{[0,t]})
\end{equation}
on the corresponding existence interval. The bound is uniform over these inputs.

\begin{corollary}\label{cor:Isidori}
Assume the global domain of the error states, global nonsingularity of $B$, and conditions on the corrected output derivatives of \cite[Sec. 5.1]{Isidori2026}. If \eqref{eq:error-effective} satisfies \eqref{eq:2026qbound}, then there is a smooth map $\mathcal R_\tau:\R^6\to\R^{n_e}$ such that $\mathcal T(w(\tau),\mathcal R_\tau(b))=b$ for every $b\in\R^6$. Thus the corrected map is onto even with the terminal exosystem state $w(\tau)$ fixed, and the separately imposed surjectivity assumption in \cite[Th. 6.1]{Isidori2026} follows from these hypotheses.
\end{corollary}
Along a lifted reference satisfying the source's homogeneous assignment near $\tau$, the correction identities give
\begin{equation}\label{eq:2026Tidentification}
 \zeta_1=e_1,\qquad \zeta_2=\col(e_2,e_2'),\qquad
 \zeta_3=\col(e_3,e_3',e_3'').
\end{equation}
Choosing the six terminal reference derivatives arbitrarily realizes every target of $\mathcal T$. Estimate~\eqref{eq:2026qbound} ensures that each lift reaches $\tau$, and smooth parameter dependence supplies $\mathcal R_\tau$. Appendix~\ref{app:regulation-proof} proves these assertions, including the bounds for $v$ and $\widetilde\eta$.

Under their global inversion and detectability hypotheses, the same conclusion holds for the three-output map in \cite[Sec. 3.2]{IsidoriNOLCOS2025} and the map in \cite[Th. 1]{IsidoriCDC2025} with any finite number $\ell$ of scalar output groups of orders $1,\ldots,\ell$. Appendix~\ref{app:2025} verifies the coordinate ordering and the general recursion. The section in Corollary~\ref{cor:Isidori} is smooth in the error coordinates $(\widetilde\eta,q)$ of the system driven by $v$. The function $\Gamma$ does not enter these error dynamics. The original output-feedback designs retain their requirements on observers, gains, and local exponential stability.

\section{Examples and Numerical Verification}\label{sec:examples}
\subsection{Verifying the state bound}\label{sec:state-verification}
The following Lyapunov condition implies Assumption~\ref{ass:B} by bounding finite-horizon growth, as in classical completeness criteria \cite{AS1999}.

\begin{proposition}\label{prop:Vcertificate}
Suppose there are a finite integer $N\ge0$, a $C^1$ function $V:\R^n\to[0,\infty)$, $\alpha_1,\alpha_2,\sigma\in\Kinf$, and constants $a,c_0\ge0$ such that
\begin{equation}\label{eq:Vproper}
 \alpha_1(|x|)\le V(x)\le\alpha_2(|x|),\qquad x\in\R^n,
\end{equation}
and, along every trajectory of \eqref{eq:plant} driven by a smooth input, throughout its interval of existence,
\begin{equation}\label{eq:Vjet}
 \frac{d}{dt}V(x(t))\le aV(x(t))+c_0+\sigma(|\jet{N}{y}(t)|).
\end{equation}
Then Assumption~\ref{ass:B} holds with
\begin{equation}\label{eq:Vbound}
 \Omega(\tau,R,M)\triangleq \alpha_1^{-1}\!\bigl(e^{a\tau}[\alpha_2(R)+\tau(c_0+\sigma(M))]\bigr).
\end{equation}
\end{proposition}
\begin{proof}
Along a maximal solution whose output $N$-jet is bounded by $M$ up to time $t$, \eqref{eq:Vjet} gives $\dot V\le aV+c_0+\sigma(M)$. For $t\le\tau$ and $|x(0)|\le R$, Gronwall's inequality bounds $V(x(t))$ by the argument of $\alpha_1^{-1}$ in \eqref{eq:Vbound}. The lower bound in \eqref{eq:Vproper} gives the claimed state estimate.
\end{proof}

The finite-horizon bound permits zero-output trajectories to grow, whereas weak uniform zero-state detectability requires convergence to zero. The coefficient $a$ may be positive because the certificate controls finite-horizon growth rather than asymptotic decay. To illustrate the distinction from detectability, in $\dot z=z$, $\dot y=u$, the proper function $V\triangleq z^2+y^2$ satisfies
\[
 \dot V=2z^2+2yu\le2V+u^2
          \le2V+|\col(y,\dot y)|^2.
\]
Hence Assumption~\ref{ass:B} holds even though the zero-output trajectory with $z(0)\ne0$ grows. The constant term also covers $\dot z=1$, $\dot y=u$, for which the same $V$ satisfies $\dot V\le V+1+|\col(y,\dot y)|^2$.

\begin{remark}
Classical estimates based on unboundedness observability may depend on both the physical input and the output \cite{AS1999}. Assumption~\ref{ass:B} bounds the state uniformly over physical inputs from a finite output jet, so the state is confined before any bound on the reconstructed input is invoked.
\end{remark}

As a further example, consider $\dot z=1$, $\dot v=u$, with output $y=v+1$. The system has no equilibrium, and Assumptions~\ref{ass:G} and~\ref{ass:F} hold with $K=1$ since $\dot y=u$. If $|(z(0),v(0))|\le R$ and $\|y\|_{[0,t]}\le M$, then
\[
 |(z(t),v(t))|\le R+\tau+M+1,\qquad 0\le t\le\tau,
\]
so Assumption~\ref{ass:B} holds with $N=0$. For the assignment $\dot y=-y$, the map and feedback are $T(z,v)=v+1$ and $\alpha(z,v)=-(v+1)$. Lifting \eqref{eq:rb} from the origin gives
\[
 R_\tau(b)=(\tau,b-1),\qquad T(R_\tau(b))=b.
\]
Thus the global section exists without an equilibrium, while $R_\tau(0)\ne0$.

\subsection{Nontriangular coupling with internal feedback}\label{sec:skew}
Let $J\in\R^{q\times q}$ be nonzero and skew-symmetric, with $q\ge2$, and fix $e,c\in\R^q$ with $|e|=1$. Take $\epsilon,\kappa\ge0$ and $a\in\R$, and put
\[
 \Phi(z)\triangleq z+\epsilon\tanh z,\qquad
 g(p,z,\eta)\triangleq -a\eta+c^Tz+p^2,
\]
where $\tanh$ acts componentwise. For states $p,\eta\in\R$ and $y,z\in\R^q$, consider
\begin{equation}\label{eq:skewplant}
\begin{split}
 \dot p&=u,\qquad \dot z=v,\qquad \dot\eta=g(p,z,\eta),\\
 \dot y&=\Phi(z)+\kappa\tanh(\eta)e+uJz,
\end{split}
\end{equation}
with outputs $(p,y)$ and inputs $(u,v)$. When $\kappa>0$ and $c\ne0$, the internal state changes the output drift and receives feedback from $z$.

Define $B_{\mathrm{inv}}(z,\xi)\triangleq D\Phi(z)+\xi J$ for $\xi\in\R$. The first stage exposes $u=\dot p$. The second gives
\begin{equation*}
 v=B_{\mathrm{inv}}(z,\dot p)^{-1}
 \bigl[\ddot y-\ddot pJz-\kappa\operatorname{sech}^2(\eta)e\,g\bigr].
\end{equation*}
For every state, every real formal $\xi$, and every $w\in\R^q$,
\begin{equation*}
 w^TB_{\mathrm{inv}}(z,\xi)w=w^TD\Phi(z)w\ge |w|^2.
\end{equation*}
Thus $B_{\mathrm{inv}}$ is nonsingular and $\|B_{\mathrm{inv}}^{-1}\|\le1$. The fixed cumulative ranks are $1$ and $1+q$. Since $L_Gy$ contains $Jz$, these original outputs do not have a uniform vector relative degree on the full state space.

With $x=\col(p,y,z,\eta)$, the identities $z^TJz=0$, $z^T\Phi(z)\ge|z|^2$, and $|\tanh\eta|\le1$ give $|z|\le|\dot y|+\kappa$. If $|x(0)|\le R$ and $\|j^1(p,y)\|_{[0,t]}\le M$ with $t\le\tau$, the upper right Dini derivative $D^+|\eta|(t)\triangleq \limsup_{\delta\downarrow0}(|\eta(t+\delta)|-|\eta(t)|)/\delta$ satisfies $D^+|\eta|\le|a||\eta|+|c|(M+\kappa)+M^2$. Integrating and using $|x|\le|p|+|y|+|z|+|\eta|$ gives
\begin{equation*}
 |x(t)|\le 3M+\kappa+e^{|a|\tau}\bigl[R+\tau\bigl(|c|(M+\kappa)+M^2\bigr)\bigr].
\end{equation*}
Thus Assumption~\ref{ass:B} holds for every $a\in\R$, and Theorems~\ref{thm:lifting} and~\ref{thm:section} apply throughout this parameter range.

If $a>\kappa|c|$, the weak uniform zero-state detectability estimate also holds. Set $M_t\triangleq \|j^1(p,y)\|_{[0,t]}$ and $\lambda\triangleq a-\kappa|c|>0$. The identities $z^TJz=0$ and $z_i\tanh z_i\ge0$ give $|z|\le|\dot y|+\kappa|\eta|$. It follows that $D^+|\eta|\le-\lambda|\eta|+|c||\dot y|+p^2$. Therefore, on the existing trajectory interval,
\begin{equation*}
 |\eta(t)|\le e^{-\lambda t}|\eta(0)|
       +\frac{|c|M_t+M_t^2}{\lambda}.
\end{equation*}
Consequently,
\[
 |x(t)|\le(1+\kappa)e^{-\lambda t}|x(0)|
        +3M_t+\frac{1+\kappa}{\lambda}(|c|M_t+M_t^2),
\]
which proves \eqref{eq:WUD} of order one, uniformly over inputs.

For $\dot p=-p$ and $\ddot y=-y-2\dot y$, put
$w\triangleq \Phi(z)+\kappa\tanh(\eta)e-pJz$ and $S\triangleq D\Phi(z)-pJ$. The map and feedback are
\begin{align*}
 T(p,y,z,\eta)&=\col(p,y,w),\\
 v_\star&=S^{-1}\bigl[-y-2w-pJz
             -\kappa\operatorname{sech}^2(\eta)e\,g\bigr],
\end{align*}
together with $u_\star=-p$. Direct differentiation gives $\dot w=-y-2w$. For a target $b=\col(b_p,b_y,b_w)$ and fixed $\eta$, set $p=b_p$ and $y=b_y$. Solving for $z$ entails the strongly monotone equation
\begin{equation}\label{eq:skewsection}
 \Phi(z)-b_pJz=b_w-\kappa\tanh(\eta)e.
\end{equation}
The lifting construction selects a state satisfying $T(x)=b$ as the endpoint of a trajectory of the inverse dynamics. For $\epsilon=\kappa=|c|=0$ and $a=1$, \eqref{eq:skewsection} gives the explicit solution $z=(I_q-b_pJ)^{-1}b_w$.

To relate this example to the correction convention in \eqref{eq:correctednonterminal}, write
\[
 \dot y=\Phi(z)+\kappa\tanh(\eta)e-K_2(x,\xi),
 \qquad K_2(x,\xi)\triangleq -\xi Jz,
\]
where $\xi$ denotes the formal value of $\dot p$. The derivative of $K_2$ in a remaining input direction $\delta v\in\R^q$ is $-\xi J\delta v$, which is not identically zero. Thus this representation fails the annihilation condition on the correction functions stated in Appendix~\ref{app:regulation-proof}. The tests in Section~\ref{sec:numerics} include the explicit special case above and the coupled family with $\epsilon,\kappa,|c|>0$.

\subsection{Why stable zero dynamics do not suffice}\label{sec:counterexample}
The next example isolates the role of Assumption~\ref{ass:B}. It satisfies the global conditions on the modified inversion algorithm and has globally exponentially stable zero dynamics, yet the map $T$ is not onto. Thus minimum-phase behavior \cite{ByrnesIsidori1988} cannot replace a state bound uniform over inputs. Finite escape with bounded output also connects to unboundedness observability \cite{AS1999}.

For states $(z,p,q,s)\in\R^4$ and inputs $(u_1,u_2)\in\R^2$, let $\theta(z)\triangleq \arctan z$ and $\delta(z,s)\triangleq \theta(z)-s$. Consider
\begin{align}
 \dot z&=(1+z^2)(u_2-\theta(z)),\label{eq:counterplant}\\
 \dot p&=u_1,\qquad
 \dot q=s-\delta(z,s)u_1,\qquad \dot s=u_2,\nonumber
\end{align}
with outputs $(p,q)$. All coefficients are globally smooth. Along every trajectory,
\begin{equation*}
 \dot\delta=-\theta(z),
\end{equation*}
so the selected input equations are
\begin{equation*}
 \dot p=u_1,\qquad
 \ddot q=u_2+\theta(z)\dot p-\delta(z,s)\ddot p.
\end{equation*}
Their cumulative input matrix is the identity. Thus the plant satisfies Assumptions~\ref{ass:G} and~\ref{ass:F} with the output equations selected at stages $1$ and $2$, for every value of the formal derivative variables.

Assign $\dot p=-p$ and $\ddot q=-q-\dot q$. The feedback and its associated map are
\begin{align*}
 u_1&=-p,\qquad u_2=-q-s+p\theta(z),\\
 T(z,p,q,s)&=\bigl(p,q,(1-p)s+p\theta(z)\bigr).
\end{align*}
When the target's first coordinate is $1$, its third coordinate must lie in $(-\pi/2,\pi/2)$. The map is therefore not onto. In fact, it is still a submersion because the derivatives of its third coordinate with respect to $(z,s)$ are $(p/(1+z^2),1-p)$ and cannot both vanish. Thus the failure of surjectivity persists even when $T$ has full differential rank everywhere.

The zero-output constraints give $p=q=s=0$ and $u_1=u_2=0$. The zero dynamics are
\begin{equation*}
 \dot z=-(1+z^2)\arctan z,
\end{equation*}
with solution and bound
\begin{equation*}
 z(t)=\tan\bigl(e^{-t}\arctan z(0)\bigr),\qquad
 |z(t)|\le e^{-t}|z(0)|.
\end{equation*}
The inequality follows from convexity of $\tan$ on $[0,\pi/2)$ and odd symmetry. Hence the zero dynamics are globally exponentially stable.

The state bound in Assumption~\ref{ass:B} fails. From $(z,p,q,s)=(0,1,0,0)$, apply the smooth inputs $u_1=0$, $u_2=t+1$. For $t<\pi/2$ the solution is
\begin{equation*}
 z(t)=\tan t,\quad p(t)=1,\quad
 s(t)=t+\tfrac12t^2,\quad
 q(t)=\tfrac12t^2+\tfrac16t^3.
\end{equation*}
Every output jet of fixed finite order stays bounded on this finite interval, while the internal state $z$ escapes. Neither Assumption~\ref{ass:B} nor \eqref{eq:WUD} holds. This proves that global regular inversion plus even global exponential stability of the zero dynamics is insufficient for surjectivity.

The arctangent output map in \cite[Example 4]{Liberzon2004} illustrates failure of surjectivity and incompleteness of the inverse vector field. The role of \eqref{eq:counterplant} here is the conjunction of a globally nonsingular input coefficient matrix, exponentially stable nontrivial zero dynamics, and failure of surjectivity of the associated output-derivative map.

Corollary~\ref{cor:escape-certificate} shows that an unattained target forces finite escape of its local lift while all finite output jets remain bounded. The converse fails. For example, $\dot y=u$, $\dot z=z^2$ has the onto map $T=y$, but $z(0)>0$ causes escape even for zero output.

\subsection{Numerical verification}\label{sec:numerics}
Numerical integration of the lifting feedback tests coordinate realization and prescribed-time tracking. Table~\ref{tab:numeric} reports results for the coupled family in Section~\ref{sec:skew}, its special case $\epsilon=\kappa=|c|=0$, $a=1$, and the three-stage model in Appendix~\ref{app:regression}. The coupled cases use $q=2$, $\epsilon=0.6$, $\kappa=0.7$, $a=2$, $c=(0.4,0.3)^T$, $e=(1,0)^T$, and $J=\left(\begin{smallmatrix}0&-1\\1&0\end{smallmatrix}\right)$. They test Corollary~\ref{cor:coordinate-maneuver} from two nonzero initial states, with four targets and horizons $\tau=0.4,1$.

\begin{table}[t]
\caption{Maximum errors in output lifting}
\label{tab:numeric}
\centering
\begin{tabular}{lccc}
\toprule
Plant & Runs & Endpoint error & Tracking error\\
\midrule
Special case of \eqref{eq:skewplant} & 12 & $1.48\times10^{-10}$ & $3.13\times10^{-11}$ \\
Model \eqref{eq:regressionplant} & 4 & $3.59\times10^{-10}$ & $1.26\times10^{-11}$ \\
Coupled system \eqref{eq:skewplant} & 16 & $5.32\times10^{-11}$ & $1.37\times10^{-11}$ \\
\bottomrule
\end{tabular}
\par\smallskip
\parbox{\columnwidth}{\footnotesize The parameter choices for the special case and the coupled system are specified in Section~\ref{sec:numerics}. Errors are maxima over the listed runs in the infinity norm. The maximum tracking error is also taken over sampled times in $[0,\tau]$.}
\end{table}

The endpoint errors in Table~\ref{tab:numeric} compare $T(x(\tau))$ with the prescribed target, and the tracking errors compare the realized output with the reference. For the coupled family, the lifted input agrees with the nominal feedback on the terminal interval to numerical accuracy. Continuation under that feedback follows the assigned coordinate evolution.

We also test Corollary~\ref{cor:prescribed-tracking} on the coupled family with the same parameters. Eight cases use sinusoidal desired outputs with four initial states and $\tau=1,2$. Two additional cases use a zero reference with $\tau=1$ and a homogeneous reference with $\tau=1.5$. The physical plant and controller generator are integrated jointly to $t=12$. The largest sampled Euclidean norm of the output error $h(x(t))-y_d(t)$ across these ten cases on $[\tau,12]$ is $1.80\times10^{-8}$ with relative and absolute tolerances $10^{-10}$ and $10^{-12}$. Tightening the tolerances reduces the maximum sampled value of $|h(x(t))-r(t)|$ on $[0,12]$ in every case.

\section{Conclusion}
Under the stated global inversion conditions, a finite-horizon state bound uniform over inputs makes the output-derivative map of every fixed homogeneous assignment surjective and gives it a smooth global section. The construction realizes each derivative target as the endpoint of a complete reference lift, including when internal dynamics are present. For the fixed output groups and assignment, the separate onto requirement in Liberzon's framework therefore follows from the inversion conditions and weak uniform zero-state detectability. For known plants, complete lifting also yields an initialized smooth controller for exact output tracking after any prescribed positive time, as well as coordinate realization followed by nominal feedback. Under weak uniform zero-state detectability, the tracking controller provides uniform state and input bounds for bounded initial states and reference derivatives of the required order. The counterexample shows that even global exponential stability of the zero dynamics, together with global regular inversion, does not ensure surjectivity.

\appendices
\section{Supporting Inversion Identities}\label{app:representation}
\subsection{Proof of Lemma~\ref{lem:representation}}
\begin{proof}
We retain every previously selected input equation. For an original output component still unselected after stage $k-1$, its derivative of order $k-1$ is a smooth function of $(x,\theta_{k-1})$. Stack all such functions into $\overline Q_{k-1}$, with $\overline Q_0\triangleq h$. The induction constructs these stacks together with the selected equations.

\emph{Differentiation.} Write the formal derivative blocks as $\theta_{j,r}\triangleq y_{[j]}^{(r)}\in\R^{m_j}$. Define
\begin{align*}
 \widehat a_k(x,\theta_k^-)
 &\triangleq D_x\overline Q_{k-1}\,f(x)
   +\sum_{j=1}^{k-1}\sum_{r=j}^{k-1}
     D_{\theta_{j,r}}\overline Q_{k-1}\,\theta_{j,r+1},
       \\
 \widehat B_k(x,\theta_{k-1})
 &\triangleq D_x\overline Q_{k-1}\,G(x).
\end{align*}
The summands for empty groups are absent. Along every plant trajectory, the unselected derivatives of order $k$ therefore equal $\widehat a_k(x,\theta_k^-)+\widehat B_k(x,\theta_{k-1})u$.
Differentiation shifts each formal variable to the next derivative of the same selected output. These derivatives remain formal arguments, so the candidate equation contains no explicit derivative of the input. This accounts for the different argument sets of $a_k$ and $B_k$ in \eqref{eq:terminal}.

\emph{Selection.} Assumption~\ref{ass:G} selects a fixed $m_k$-row block of these candidate equations, defining $a_k,B_k$ and the newly selected output group $y_{[k]}$. Earlier selected input equations remain at their original derivative orders. Using the same formal derivative vector $\theta_{k-1}$ in every block, put
\begin{equation*}
 \BB^{[k]}\triangleq \col_{j=1}^{k}B_j(x,\theta_{j-1}).
\end{equation*}
This matrix belongs to $\R^{\rho_k\times m}$ and has $\rho_k$ independent rows. The row block $\widetilde B_k\in\R^{(m-\rho_k)\times m}$ belonging to the outputs still unselected after stage $k$ lies in its row span. If $\rho_k>0$, define $P_k\in\R^{(m-\rho_k)\times\rho_k}$ by
\begin{equation*}
 P_k\triangleq \widetilde B_k(\BB^{[k]})^\top
       \bigl(\BB^{[k]}(\BB^{[k]})^\top\bigr)^{-1}.
\end{equation*}
Thus $\widetilde B_k=P_k\BB^{[k]}$. The Gram matrix is positive definite at every state and every value of the formal derivative variables, so $P_k$ is globally smooth. This formula defines the elimination coefficients without choosing a column minor.

\emph{Elimination.} Let $\widetilde a_k$ denote the unselected part of $\widehat a_k$. Substituting the retained input equations into $\widetilde a_k+P_k\BB^{[k]}u$ gives
\begin{equation}\label{eq:Qrecursion}
\begin{aligned}
 \overline Q_k(x,\theta_k)\triangleq {}&\widetilde a_k(x,\theta_k^-)
       +P_k(x,\theta_{k-1})\\
 &{}\times\col_{j=1}^{k}\bigl[\theta_{j,j}-a_j(x,\theta_j^-)\bigr].
\end{aligned}
\end{equation}
All arguments in this expression are restrictions of the same $\theta_k$. When $\rho_k=0$, the candidate input matrix is zero, and \eqref{eq:Qrecursion} means simply $\overline Q_k=\widehat a_k$. When $m_k=0$ but $\rho_k>0$, the same elimination uses the retained stack and differentiation proceeds without selecting a new group.

For an output selected at stage $k$, its derivative of each order $\ell<k$ is the corresponding component of $\overline Q_\ell$, defining $Q_{k,\ell}$. Equations \eqref{eq:terminal} and \eqref{eq:nonterminal} follow by induction. At stage $K$, $\BB^{[K]}=\BB$ has $m$ independent rows and $m$ columns, proving its global nonsingularity.

Factorization through the newly selected rows alone, as printed after (18) in \cite{Liberzon2004}, is valid when the residual rows belong to their span. Here the factorization uses all retained rows, as in \cite[Sec. 4]{IsidoriJSSC2025}. For the same fixed row selections, the two recursions coincide if the residual rows lie in the span of the newly selected rows at every stage and every formal argument. Indeed, independence of the retained rows then forces the coefficients of all earlier rows to vanish. The elimination formulas agree, and induction gives the same subsequent equations and assigned map.

The construction uses differentiation, fixed row permutations, and inverses of globally nonsingular matrices. All identities are defined for all values of the formal derivative variables before any realizability claim. Theorem~\ref{thm:lifting} establishes realizability for compatible references.
\end{proof}

\subsection{Proof of Lemma~\ref{lem:rigidity}}\label{app:assignment}
\begin{proof}
For each $\ell$, the formal derivative variables will be assigned the values
\begin{equation}\label{eq:homogeneous-sheet}
 \Theta_\ell(x)\triangleq \col_{1\le j\le r\le\ell}
                 C_jA_j^rT_{[j]}(x).
\end{equation}
Thus $\Theta_\ell(x)$ assigns values depending on the state to the formal variables $\theta_\ell$, in the same order as $\DD_\ell(r)$ extracts them from a reference. The corresponding $\Theta_k^-(x)$ omits group $k$. Define the state blocks successively by
\begin{equation}\label{eq:Trecursive}
 T_{[k]}(x)\triangleq \col_{0\le\ell<k}
              Q_{k,\ell}\bigl(x,\Theta_\ell(x)\bigr).
\end{equation}
For $\ell<k$, the vector $\Theta_\ell$ uses only the previously defined groups $j\le\ell<k$, so the recursion is well defined. If earlier groups are empty, this vector is empty. The first $m_k$ entries of $T_{[k]}$ form $h_{[k]}$.

Once all blocks are defined, set
\begin{equation}\label{eq:alpha-recursive}
\begin{aligned}
 \alpha(x)\triangleq {}&\BB\bigl(x,\Theta_{K-1}(x)\bigr)^{-1}\\
 &{}\times\col_k\left[C_kA_k^kT_{[k]}(x)-a_k(x,\Theta_k^-(x))\right].
\end{aligned}
\end{equation}
Every row uses the corresponding components of the same vector in \eqref{eq:homogeneous-sheet}. Lemma~\ref{lem:representation} therefore ensures that \eqref{eq:alpha-recursive} defines a globally smooth feedback $\alpha$. 

Consider the local trajectory of $F=f+G\alpha$ starting from an arbitrary state. For the first nonempty group, \eqref{eq:nonterminal} identifies its output derivatives below the selection order with \eqref{eq:Trecursive}. The corresponding row in \eqref{eq:alpha-recursive} imposes \eqref{eq:assignment}. Hence all of that group's higher derivatives have the form
\begin{equation}\label{eq:homogeneousjets}
 y_{[j]}^{(r)}=C_jA_j^rT_{[j]}(x),\qquad r\ge0.
\end{equation}
Suppose this has been established for all earlier groups. Substituting \eqref{eq:homogeneousjets} into the next group's identities \eqref{eq:nonterminal} gives its output derivatives below the selection order as \eqref{eq:Trecursive}. The row for the next group then imposes its homogeneous equation. Induction proves the claim for every group, including multidimensional groups. Differentiating the resulting derivative stacks proves \eqref{eq:semiconj} and \eqref{eq:Tdef}. Since the initial state was arbitrary, these are pointwise identities on all of $\R^n$.

If $f(0)=0$ and $h(0)=0$, the zero-input equilibrium at $x=0$ makes every actual output derivative zero. Thus $Q_{k,\ell}(0,0)=0$ and $a_k(0,0)=0$. Induction in \eqref{eq:Trecursive} gives $T(0)=0$, and the unique solution of \eqref{eq:alpha-recursive} is $\alpha(0)=0$.

Finally take any plant trajectory driven by a smooth input and satisfying \eqref{eq:assignment} on an interval, without assuming its input equals $\alpha$. The first group's output derivatives below the selection order are its state expressions from \eqref{eq:nonterminal}, and its higher ones obey \eqref{eq:homogeneousjets}. They identify the next block by the same recursion \eqref{eq:Trecursive}. Repeating this argument identifies all blocks and all formal arguments of the stacked input coefficient matrix. The selected input equations now coincide with \eqref{eq:alpha-recursive}, whose solution is unique. Hence the trajectory input is $\alpha(x)$ on that interval.
\end{proof}

\section{Quantitative Properties of Output Lifts}\label{app:quantitative}
Using the smooth inversion map $\mathcal U$ defined in Section~\ref{sec:liftdependence}, define the lifting vector field for every formal $K$-jet $\mathbf j\in\R^{m(K+1)}$ by
\begin{equation*}
 \mathcal V(x,\mathbf j)\triangleq f(x)+G(x)\mathcal U(x,\mathbf j).
\end{equation*}
These globally smooth maps give quantitative bounds for compatible reference data on bounded sets.
\begin{proposition}\label{prop:lift-stability}
Set $\nu\triangleq \max\{K,N\}$ to include the derivative orders used by inversion and the state bound. Fix a finite horizon $\tau>0$ and finite bounds $R,M\ge0$. Under Assumptions~\ref{ass:G}--\ref{ass:B}, consider two compatible data pairs $(x_0,r)$ and $(\widetilde x_0,\widetilde r)$ satisfying
\begin{equation*}
 |x_0|,|\widetilde x_0|\le R,\qquad
 \|\jet{\nu}{r}\|_{[0,\tau]},
 \|\jet{\nu}{\widetilde r}\|_{[0,\tau]}\le M.
\end{equation*}
Let $(x,u)$ and $(\widetilde x,\widetilde u)$ denote their respective lifts. There are finite nonnegative constants $L_x,L_j,L_u$, depending only on these bounds and the plant, such that, for $0\le t\le\tau$,
\begin{align}
 |x(t)-\widetilde x(t)|
 &\le e^{L_xt}|x_0-\widetilde x_0|
       +L_jt e^{L_xt}\varepsilon_r,\label{eq:lipschitz-lifts}\\
 |u(t)-\widetilde u(t)|
 &\le L_u\bigl(|x(t)-\widetilde x(t)|+\varepsilon_r\bigr),
       \label{eq:lipschitz-inputs}
\end{align}
where $\varepsilon_r\triangleq \|\jet{K}{(r-\widetilde r)}\|_{[0,\tau]}$.
\end{proposition}
\begin{proof}
Both states lie in the ball supplied by Lemma~\ref{lem:continuation}. The maps $\mathcal U$ and $\mathcal V$ are smooth on the full state and formal jet domain. Their first derivatives have finite maxima on the product of a slightly larger closed ball in the state space and the closed ball of radius $M+1$ in the space of formal $K$-jets. The mean-value estimate gives
\[
 |\mathcal V(x,\mathbf j)-\mathcal V(\widetilde x,\widetilde{\mathbf j})|
 \le L_x|x-\widetilde x|+L_j|\mathbf j-\widetilde{\mathbf j}|.
\]
Integrating this inequality and applying Gronwall yields \eqref{eq:lipschitz-lifts}. Applying the analogous derivative bound for $\mathcal U$ yields \eqref{eq:lipschitz-inputs}. The maps are globally defined for all formal derivative variables, so the derivative bounds apply even along line segments that do not correspond to compatible data.
\end{proof}

Let $\lambda$ vary in an open subset of $\R^a$, and let $(x_0(\lambda),r_\lambda)$ be a jointly smooth family of compatible data with state lifts $x_\lambda$. Set $Z(t)\triangleq D_{\lambda}x_\lambda(t)$ and $\mathbf j_\lambda(t)\triangleq j^Kr_\lambda(t)$. Write $\mathcal V_x\triangleq D_x\mathcal V$ and $\mathcal V_{\mathbf j}\triangleq D_{\mathbf j}\mathcal V$. Differentiating the lifting equation gives
\begin{equation}\label{eq:variational-lift}
 \dot Z=\mathcal V_x(x_\lambda,\mathbf j_\lambda)Z
          +\mathcal V_{\mathbf j}(x_\lambda,\mathbf j_\lambda)D_{\lambda}\mathbf j_\lambda,
 \qquad Z(0)=D_{\lambda}x_0(\lambda).
\end{equation}
All coefficients are bounded on compact parameter neighborhoods. This formula supplies the sensitivity equation for the section constructed in Section~\ref{sec:geometry}.

For each compact target set, Proposition~\ref{prop:lift-stability} gives a finite local Lipschitz constant for $R_\tau$, because the reference in \eqref{eq:rb} is affine in $b$. For this family, take $\lambda=b$ and $x_0(\lambda)=0$. The derivative of the section is obtained by evaluating \eqref{eq:variational-lift} at $t=\tau$ with $Z(0)=0$. The sensitivity bound is local in $b$ and depends on the cutoff and horizon.

\section{Verification for the Feedback Applications}\label{app:regulation}
\subsection{Output regulation}\label{app:regulation-proof}
We use the state and correction functions introduced in Section~\ref{sec:2026}.
For $x=\col(w,\widetilde\eta,q)$, the drift and input matrix of the augmented error system are
\[
\begin{aligned}
 f_{\mathrm a}(x)&\triangleq \col\bigl(s(w),M_\eta\widetilde\eta,f_e(w,q)\bigr),\\
 G_{\mathrm a}(x)&\triangleq \col\bigl(0,N_\eta,G_e(w,q)\bigr).
\end{aligned}
\]
For a smooth function $K$ of $x$ and finitely many formal variables $\xi_{j,r}$ representing $e_j^{(r)}$, define
\[
 \mathscr D_0K\triangleq D_xK\,f_{\mathrm a}(x)
       +\sum_{j,r}D_{\xi_{j,r}}K\,\xi_{j,r+1}.
\]
The sum runs over the formal arguments of $K$, with derivatives taken while holding the other arguments fixed. This is the operator denoted $L_F$ in \cite{Isidori2026}. The input matrix denoted $g$ there is $G_{\mathrm a}$ here. With the source's coefficient functions $F_{11}\in\R^2$ and $F_{21},F_{22}\in\R$, the correction identities are
\begin{align}
 K_2&=F_{11}e_1',\label{eq:correctedK2}\\
 K_3&=F_{21}e_1'+F_{22}(e_2''+\mathscr D_0K_2^{(2)})
                     +\mathscr D_0K_2^{(3)}.\label{eq:correctedK3}
\end{align}
The source's conditions $L_gK_2=0$ and $L_gK_3=0$ therefore read $D_xK_2\,G_{\mathrm a}=0$ and $D_xK_3\,G_{\mathrm a}=0$, with the formal variables held fixed. Writing $a_i$ for the source's drift terms and using the input rows $b_i$ defined in Section~\ref{sec:2026}, the selected input equations become
\begin{align}
 e_1'&=a_1+b_1v,\label{eq:corrected1}\\
 e_2''&=a_2+b_2v-\mathscr D_0K_2^{(2)},\nonumber\\
 e_3'''&=a_3+b_3v-\mathscr D_0K_3.\nonumber
\end{align}
To lift $e=r$, define
\begin{equation}\label{eq:2026lift}
 v_r\triangleq B^{-1}\left\{
 \begin{pmatrix}
 r_1'\\
 r_2''+\mathscr D_0K_2^{(2)}(x,j^3r(t))\\
 r_3'''+\mathscr D_0K_3(x,j^3r(t))
 \end{pmatrix}-\begin{pmatrix}a_1\\a_2\\a_3\end{pmatrix}
 \right\}.
\end{equation}
Here $j^3r(t)$ collects the reference values and derivatives through order three, substituted into every correction term throughout the lifting interval. The homogeneous assignment enters when the terminal reference jets are identified with the corrected coordinates.

\begin{proof}[Proof of Corollary~\ref{cor:Isidori}]
Take a reference zero near the initial time and start at $q=0$, $\widetilde\eta=0$, with fixed $w(0)\in W$. Equation~\eqref{eq:corrected1} and \eqref{eq:2026lift} first give $e_1'=r_1'$. Substituting the resulting equality of all required derivatives of the first output group gives $e_2''=r_2''$, and then $e_3'''=r_3'''$. Equation~\eqref{eq:correctednonterminal} supplies the matching initial derivatives. Thus $e=r$ on the maximal local interval.

Since $e=r$, estimate~\eqref{eq:2026qbound} bounds $q$ in terms of the reference derivatives. The smooth extension of each realized input segment in Lemma~\ref{lem:continuation} justifies using this estimate before global existence has been established. It bounds $q$ on a fixed finite horizon. Since $w(t)\in W$ and $B^{-1}$ and all coefficients in the effective input equations are bounded on the resulting compact set, \eqref{eq:2026lift} bounds $v_r$ as well. Finally,
\begin{equation*}
 \widetilde\eta(t)=e^{M_\eta t}\widetilde\eta(0)
          +\int_0^t e^{M_\eta(t-s)}N_\eta v_r(s)\,ds
\end{equation*}
is finite on that horizon. Continuation therefore completes the lift.

Let $H_i$ be the companion matrices selected in the source, $H\triangleq \diag(H_1,H_2,H_3)$, and let $C$ select the first entry of each of the three blocks. Prescribe $r_b\triangleq \chi Ce^{H(t-\tau)}b$ for $b\in\R^6$. Near $\tau$, the first output satisfies $e_1'=H_1e_1$. Therefore the corrected second coordinate in the source's $\zeta_2$ is exactly $e_2'$, by \eqref{eq:correctednonterminal}. Its reference satisfies the second assigned equation, identifying the formal $e_2''$ that enters $K_3$. Repeating the substitution in $K_2^{(3)}$ and $K_3$ gives \eqref{eq:2026Tidentification}. These identities recover the source's recursively defined coordinates. Their values at $\tau$ are the arbitrary blocks of $b$. If $(\widetilde\eta_b,q_b)$ is the error part of the lift, define $\mathcal R_\tau(b)\triangleq \col(\widetilde\eta_b(\tau),q_b(\tau))$. Smooth dependence on $b$ supplies the smooth section, with $\mathcal T(w(\tau),\mathcal R_\tau(b))=b$. The zero target gives the zero-error endpoint at this same $w(\tau)$.
\end{proof}

\subsection{State feedback maps}\label{app:2025}
We verify the two applications stated in Section~\ref{sec:2026} in their original plant coordinates, with dynamics $\dot x=f(x)+g(x)u$ and output $y$. In this subsection, $\mathscr D_0$ uses the same formula with $f_{\mathrm a}=f$ and input matrix $g$. Both sources assume a zero equilibrium, global nonsingularity of their stacked input coefficient matrix, and weak uniform zero-state detectability. Consequently, once a smooth reference has been locally lifted from the origin, Lemma~\ref{lem:continuation} extends its lift to $\tau$.

For the three-output construction in \cite{IsidoriNOLCOS2025}, write $F_{11}=\col(F_{11}^{(2)},F_{11}^{(3)})$. The correction functions are \eqref{eq:correctedK2}--\eqref{eq:correctedK3}, with $e=y$, superscripts $(2),(3)$ denoting the two components, and $\mathscr D_0$ computed along the plant drift. The coefficient conditions in \cite[eq. (10)]{IsidoriNOLCOS2025} imply $L_gK_2=0$ and $L_gK_3=0$ for all formal derivatives, for the plant input matrix $g$. Indeed, expanding $K_3$ gives coefficients
\[
 F_{21}+F_{22}L_fF_{11}^{(2)}+L_fF_{11}^{(3)},\quad
 F_{22}F_{11}^{(2)}+F_{11}^{(3)},\quad F_{22}
\]
of $y_1',y_1'',y_2''$, respectively. Their Lie derivatives along the input vector fields vanish because
$L_gF_{11}=0$, $L_gF_{22}=0$, $L_gL_fF_{11}^{(2)}=0$, and
$L_gL_fF_{11}^{(3)}+L_gF_{21}=0$.
Thus \eqref{eq:2026lift}, with $u$ in place of $v$, supplies the local lift.
Near $\tau$, impose the source's three homogeneous output equations. Successive substitution into its definitions of the coordinates $\xi_{ij}$ in Section~3.1 gives
$\xi_{11}=y_1$, $\xi_{21}=y_2$, $\xi_{31}=y_3$,
$\xi_{22}=y_2'$, $\xi_{32}=y_3'$, and $\xi_{33}=y_3''$.
The map in its equation (16), denoted $T_{\mathrm N}$ here, therefore satisfies
\[
 T_{\mathrm N}(x(\tau))
 =P\col(y_1,y_2,y_2',y_3,y_3',y_3'')(\tau),
\]
where $P(z_1,z_2,z_3,z_4,z_5,z_6)^\top
\triangleq (z_1,z_2,z_4,z_3,z_5,z_6)^\top$ for $z\in\R^6$.
For a target $c\in\R^6$, choose the grouped terminal derivatives as $P^{-1}c$. Complete lifting then realizes $T_{\mathrm N}(x(\tau))=c$ and gives a smooth section by dependence on $c$.

For \cite{IsidoriCDC2025}, use its $\ell$ scalar output groups of orders $1,\ldots,\ell$, and set $d\triangleq \sum_{i=1}^{\ell}i$. Write $X_i=\col(X_{i,1},X_{i,\mathrm{rem}})$, where $X_{i,1}$ is scalar and $X_{i,\mathrm{rem}}$ contains the remaining $\ell-i$ components. Split the correction $K_i=\col(K_{i,1},K_{i,\mathrm{rem}})$ in the same way, starting with $K_1\triangleq 0$. The subscript $1$ denotes the first component and $\mathrm{rem}$ the remaining components. Let $a_i\triangleq L_fX_{i,1}$, $b_i\triangleq L_gX_{i,1}$, and $F_{ij}\in\R^{\ell-i}$, $1\le j\le i<\ell$, be the source's coefficients. Its defining identities give, for $i<\ell$,
\[
 L_gX_{i,\mathrm{rem}}+\sum_{j=1}^iF_{ij}b_j=0,\qquad
 L_fX_{i,\mathrm{rem}}+\sum_{j=1}^iF_{ij}a_j=X_{i+1}.
\]
With $\mathscr D_0$ again denoting drift differentiation plus formal derivative shifting, differentiation of the remaining output equations yields
\begin{equation}\label{eq:CDCcorrection}
 K_{i+1}=\sum_{j=1}^iF_{ij}
 \bigl(y_j^{(j)}+\mathscr D_0K_{j,1}\bigr)+\mathscr D_0K_{i,\mathrm{rem}}.
\end{equation}
In fact, putting $y_{i:\ell}\triangleq \col(y_i,\ldots,y_\ell)$, induction from $X_1=y$ gives
\begin{equation}\label{eq:CDCidentities}
 y_{i:\ell}^{(i-1)}=X_i-K_i,\qquad
 y_i^{(i)}=a_i+b_i u-\mathscr D_0K_{i,1}.
\end{equation}
The induction uses the source's condition $L_gK_i=0$ for all formal arguments \cite[eq. (7)]{IsidoriCDC2025}. Differentiating the tail of the first identity and substituting $a_j+b_j u=y_j^{(j)}+\mathscr D_0K_{j,1}$ gives \eqref{eq:CDCcorrection} and the first identity in \eqref{eq:CDCidentities} at stage $i+1$. Differentiating its first component gives the next selected input equation.

Set $B\triangleq \col_{i=1}^{\ell}b_i$. Substituting a reference $r$ into the formal arguments defines the smooth inverse input
\[
 u_r(t,x)\triangleq B(x)^{-1}\col_{i=1}^{\ell}
 \bigl[r_i^{(i)}(t)-a_i(x)
       +\mathscr D_0K_{i,1}(x;j^\ell r(t))\bigr].
\]
Each $K_i$ uses only derivatives of earlier output groups. The same tracking induction therefore gives $y=r$ from the zero initial state when $r$ is zero near $0$, and continuation completes the lift.

Let $H_i$ be the companion matrix selected for output $i$ in the source and $C_i$ the selector of its first coordinate. For an arbitrary target $b=\col_{i=1}^{\ell}b_{[i]}\in\R^d$, with $b_{[i]}\in\R^i$, choose
$r_i(t)\triangleq \chi(t)C_i e^{H_i(t-\tau)}b_{[i]}$.
The source constructs its coordinates layer by layer by substituting the derivatives prescribed by earlier homogeneous groups into $X_i-K_i$. Near $\tau$, these substitutions equal the actual reference derivatives. Equation~\eqref{eq:CDCidentities} consequently identifies layer $i$ with $y_{i:\ell}^{(i-1)}$, starting with $X_1=y$. Thus the map $T_{\mathrm C}$ in \cite[eq. (24)]{IsidoriCDC2025}, ordered by output group, satisfies
\[
 T_{\mathrm C}(x(\tau))
 =\col_{i=1}^{\ell}(y_i,y_i',\ldots,y_i^{(i-1)})(\tau)
 =b.
\]
This proves surjectivity for every finite $\ell$ when each output group is scalar. Smooth dependence of the lift on the terminal data gives the corresponding section.

\section{An Explicit Three-Stage Verification Model}\label{app:regression}
For scalar state coordinates $p,q,z,a,b,c,\eta$ and inputs $u_1,u_2,u_3$, the three-stage model used in the verification code is
\begin{align}
 \dot p&=u_1,\quad \dot q=z+q^2u_1,\quad \dot z=u_2,\label{eq:regressionplant}\\
 \dot a&=b+zu_1,\quad
 \dot b=c+bu_1+qu_2,\quad \dot c=u_3,\nonumber\\
 \dot\eta&=-\eta+p^2,\qquad y=\col(p,q,a).\nonumber
\end{align}
Its output equations are selected at stages $1$, $2$, and $3$, respectively. Write the reference as $r=\col(r_p,r_q,r_a)$, matching the three output components.

For \eqref{eq:regressionplant}, define
\begin{align*}
 d_q&\triangleq z+q^2u_1,\\
 \psi&\triangleq r_q''-2q d_q u_1-q^2\dot u_1.
\end{align*}
Choose $u_1=r_p'$ and $u_2=\psi$. Then $\dot q=d_q$ and $\ddot q=r_q''$ identically. Along this feedback,
\begin{equation*}
 \dot\psi=r_q'''-2d_q^2u_1-2qr_q''u_1-4q d_q\dot u_1-q^2\ddot u_1.
\end{equation*}
The remaining output satisfies
\begin{equation*}
 \ddot a=c+\phi,\qquad
 \phi\triangleq bu_1+(q+u_1)\psi+z\dot u_1.
\end{equation*}
Using $\dot b=c+bu_1+q\psi$, the derivative of $\phi$ is
\begin{equation}\label{eq:regphidot}
 \dot\phi=(c+bu_1+q\psi)u_1+b\dot u_1+(d_q+2\dot u_1)\psi
                     +(q+u_1)\dot\psi+z\ddot u_1.
\end{equation}
Hence $u_3=r_a'''-\dot\phi$ gives $a'''=r_a'''$. The coefficient of each newly determined input is one after substituting the derivatives of the previously selected outputs. This verifies Assumptions~\ref{ass:G} and~\ref{ass:F} for the modified inversion algorithm in the stated order.

The lower order derivative equations give explicit reconstruction:
\begin{align*}
 z&=q'-q^2p',\\
 u_2&=q''-2qq'p'-q^2p'',\\
 b&=a'-zp',\\
 c&=a''-bp'-(q+p')u_2-zp''.
\end{align*}
Thus $|(p,q,z,a,b,c)|$ is bounded by a polynomial in $|\jet{2}{y}|$ with zero constant term. The internal equation gives $|\eta(t)|\le e^{-t}|\eta(0)|+\|p\|_{[0,t]}^2$. Combining these bounds proves \eqref{eq:WUD} for this example.

For $\dot p=-p$, $\ddot q=-q-2\dot q$, and $a'''=-a-3a'-3a''$, put
\begin{align*}
 d_{q,*}&\triangleq z-pq^2,\\
 \psi_*&\triangleq -q-2d_{q,*}+2pq d_{q,*}-pq^2,\\
 \phi_*&\triangleq -pb+(q-p)\psi_*+pz.
\end{align*}
The output-derivative map used in the computation is
\begin{equation*}
 T=\col\bigl(p,q,d_{q,*},a,b-pz,c+\phi_*\bigr).
\end{equation*}
In \eqref{eq:regphidot}, substituting $u_1=-p$, $\dot u_1=p$, $\ddot u_1=-p$, $r_q''=-q-2d_{q,*}$, and $r_q'''=-d_{q,*}-2r_q''$ gives the third nominal feedback by assigning $r_a'''=-a-3(b-pz)-3(c+\phi_*)$. Direct differentiation verifies $DTF=AT$. The verification script integrates this model from the origin for four prescribed targets using the reference family constructed with the cutoff function, including the transition before the terminal homogeneous interval.

\bibliographystyle{IEEEtran}
\bibliography{references}

\end{document}